\documentclass{article}

\usepackage{arxiv}

\usepackage[utf8]{inputenc} 
\usepackage[T1]{fontenc}    
\usepackage{hyperref}       
\usepackage{url}            
\usepackage{booktabs}       
\usepackage{amsfonts}       
\usepackage{nicefrac}       
\usepackage{microtype}      
\usepackage{lipsum}		
\usepackage{graphicx}
\usepackage[numbers]{natbib}
\usepackage{doi}
\usepackage{enumerate}
\usepackage{amsmath}
\usepackage{amsthm}
\usepackage{xcolor}
\usepackage{amssymb}
\usepackage{longtable}

\usepackage{tikz}
\usetikzlibrary{arrows}
\usetikzlibrary[topaths]
\usepackage{cancel}

\usepackage{tikz-cd}
\usetikzlibrary{cd}
\usetikzlibrary{shapes,snakes}

\usepackage{subcaption}

\theoremstyle{definition}
\newtheorem{satz}{Satz}[section] 
\newtheorem{lemma}[satz]{Lemma}
\newtheorem{korollar}[satz]{Corollary}
\newtheorem{proposition}[satz]{Proposition}
\newtheorem{remark}[satz]{Remark}
\newtheorem{theorem}[satz]{Theorem}

\newtheorem{definition}[satz]{Definition} 
\newtheorem{beispiel}[satz]{Example} 

\newcommand{\N}{\mathbb{N}}

\newcommand{\C}{\mathbb{C}}
\newcommand{\T}{\mathbb{T}}

\DeclareMathOperator{\id}{id}

\title{Hypergraph $C^*$-algebras}

\author{ 
Mirjam Trieb\\
    \And
    \href{https://www.uni-saarland.de/lehrstuhl/weber-moritz.html}{\hspace{1mm}Moritz Weber} \\
	Department of Mathematics,
	Saarland University\\
	\url{weber@math.uni-sb.de}
	\And 
	Dean Zenner\\
	\url{zenner@math.uni-sb.de}
}

\renewcommand{\headeright}{}
\renewcommand{\undertitle}{}
\renewcommand{\shorttitle}{Hypergraph $C^*$-algebras}

\begin{document}
\maketitle

\begin{abstract}
	We give a definition of hypergraph $C^*$-algebras. These generalize the well-known graph $C^*$-algebras as well as ultragraph $C^*$-algebras. In contrast to those objects, hypergraph $C^*$-algebras are not always nuclear. We provide a number of non-nuclear examples, we prove a Gauge-Invariant Uniqueness Theorem for a subclass of hypergraph $C^*$-algebras and we study moves on hypergraphs which generalize the moves in the theory of graph $C^*$-algebras.
\end{abstract}


\section{\textbf{Introduction}}
In 1980, Cuntz-Krieger algebras \cite{CuntzKrieger} were introduced which were later extended to the more general class of graph $C^*$-algebras. Given a graph, the idea is to associate a universal $C^*$-algebra generated by projections corresponding to the vertices and partial isometries corresponding to the edges, satisfying certain relations coming from the graph, see below. Graph $C^*$-algebras form a huge and very important class of examples in the theory of $C^*$-algebras and they are very well understood. We refer to \cite{Raeburn} for an overview on this topic. Recently, Eilers, Restorff, Ruiz and S\o rensen \cite{complete_classification, geometric_finite_graphs} classified graph $C^*$-algebras by means of K-theory. Their result is a big breakthrough in the theory of graph $C^*$-algebras.
 There have been several attempts to generalize graph $C^*$-algebras. In particular, we want to mention the work by Tomforde \cite{Tomforde} on ultragraph $C^*$-algebras. Furthermore, there is related work on higher rank graph $C^*$-algebras \cite{higherrank}, Exel-Laca algebras \cite{ExelLaca}, Cuntz-Pimsner algebras \cite{CuntzPimsner}, 
topological graph $C^*$-algebras \cite{Katsura}, 
 edge separated graph algebras \cite{Ara, Duncan},
 Quantum Cuntz-Krieger algebras \cite{QuantumCuntzKrieger1, QuantumCuntzKrieger2}, just to name a few.

 In this article, we propose a definition of hypergraph $C^*$-algebras (note that Fritz \cite{Fritz} also gave a definition of a hypergraph $C^*$-algebra which is completely different from ours). A continuation of the present article is the work by Sch\"afer and the second author \cite{Schaefer}, on nuclearity aspects, and the work by Faro\ss\; \cite{faross} on quantum symmetries.
 Our objects are in their syntax relatively close to graph and ultragraph $C^*$-algebras. 
 
 A (finite) hypergraph $H\Gamma=(V,E,r,s)$ is given by a (finite) set of vertices $V$, a (finite) set of edges $E$ and source and range maps $s,r:E\to \mathcal P(V)\backslash\{\emptyset\}$, where $\mathcal P(V)$ is the powerset of $V$. In this framework, a graph is a hypergraph with the property $\lvert s(e)\rvert=\lvert r(e)\rvert=1$ for all edges $e\in E$. In other words, when passing from graphs to hypergraphs, we replace the set $V$ for $s,r:E\to V$ by the powerset $\mathcal P(V)\backslash \{\emptyset\}$. An ultragraph is a hypergraph with the property $\lvert s(e)\rvert =1$ for all edges $e\in E$ (and $\lvert r(e)\rvert\geq 1$). In this article, we mostly restrict to the case of finite hypergraphs. 
 
 Here is an overview of the definitions of graph $C^*$-algebras, ultragraph $C^*$-graphs (in our language) and our definition of hypergraph $C^*$-algebras: Given a finite hypergraph $H\Gamma=(V,E,r,s)$, we consider the universal $C^*$-algebra generated by mutually orthogonal projections $p_v$, $v\in V$ and partial isometries $s_e$, $e\in E$ such that the following relations hold.

 \begin{tabular}{ll|c|c|c}
&
      &graph $C^*$-algebra
      &ultragraph $C^*$-algebra
      &hypergraph $C^*$-algebra  \\\hline
\multicolumn{2}{l|}{restrictions on $r$ and $s$}
      & $\lvert s(e)\rvert=\lvert r(e)\rvert =1$
      & $\lvert s(e)\rvert=1$
      &none\\
Relation (1) &$\forall e,f\in E$
      &$s_e^*s_f=\delta_{ef}p_{r(e)}$
      &$s_e^*s_f=\delta_{ef}\sum_{v\in r(e)}p_v$
      &$s_e^*s_f=\delta_{ef}\sum_{v\in r(e)}p_v$\\
Relation (2a) &$\forall e\in E$
      &
      &$s_es_e^*\leq p_{s(e)}$ 
      & $s_es_e^*\leq\sum_{v\in s(e)}p_v$   \\
Relation (2b) &$\forall v \in V:s^{-1}(v)\neq\emptyset$
      &$p_v =\sum_{\substack{e\in E\\s(e)=v}}s_es_e^*$
      &$p_v \leq\sum_{\substack{e\in E\\v\in s(e)}}s_es_e^*$ 
      &$p_v \leq\sum_{\substack{e\in E\\v\in s(e)}}s_es_e^*$ 
 \end{tabular}

\pagebreak

A generalization of graph $C^*$-algebras is desirable for  various reasons:
\begin{enumerate}
    \item Graph $C^*$-algebras behave nicely with respect to their combinatorial data -- much of the structure of the $C^*$-algebra may be read from the underlying graph. Hence, extending the class by generalizing slightly the underlying combinatorial objects  promises to produce a class with similar nice behaviour.
    \item The class of graph $C^*$-algebras does not satisfy the ``2 out of 3'' property: Given a short exact sequence of $C^*$-algebras, two of which are graph $C^*$-algebras, the third one might fail to be in that class.
    \item The class of graph $C^*$-algebras restricts to nuclear $C^*$-algebras.
    \item The class of graph $C^*$-algebras restricts to $C^*$-algebras whose $K_1$-groups are free.
    \item Since the class of graph $C^*$-algebras has been completely classified, it is time for the next, larger class.
\end{enumerate}
    
The article is organized as follows. Section \ref{sec:headings} contains preliminaries on universal $C^*$-algebras and on graph $C^*$-algebras. In Section \ref{sec:others}, we define hypergraph $C^*$-algebras (for finite hypergraphs, but as a side remark also for infinite hypergraphs), we show that they generalize graph $C^*$-algebras and ultragraph $C^*$-algebras, we show that they are always unital, we give a number of easy examples, we discuss the problem of defining paths in hypergraphs and we study the decomposition of ranges. In Section \ref{sec:nonnuclear}, we give a number of examples of non-nuclear hypergraph $C^*$-algebras such as $C(S^1)*\mathbb C^n$ and $\mathcal O_m *\mathbb C^n$, and we show how to construct a number of further such examples. Recall that all graph $C^*$-algebras and all ultragraph $C^*$-algebras are nuclear, so the class of hyperclass $C^*$-algebras is substantially larger. In Section \ref{sec:gauge}, we investigate gauge uniqueness for hypergraph $C^*$-algebras: A Gauge-Invariant Uniqueness Theorem holds for all graph $C^*$-algebras and all ultragraph $C^*$-algebras, but it fails to be true in general for hypergraph $C^*$-algebras (we give a counterexample). However, for a certain subclass of hypergraphs, we may prove a Gauge-Invariant Uniqueness Theorem in Section \ref{sec:gauge}. In Section \ref{sec:moves}, we study moves on hypergraphs. Moves on graphs play an important role in the classification of graph $C^*$-algebras as they preserve the $C^*$-algebra up to Morita equivalence. We define moves S, R, O and I on hypergraphs and we show that for a subclass of hypergraphs the first three moves yield some weakening of Morita equivalence, whereas the latter one yields isomorphic $C^*$-algebras. We end with a number of open questions  in Section \ref{Further_Research}.  In the appendix \ref{Appendix_A} and\ref{Appendix_Non_Amenable_hypergraphs}, we list further examples of hypergraph $C^*$-algebras, nuclear and non-nuclear ones.

\section*{\textbf{Acknowledgements}}

The definition of hypergraph $C^*$-algebras was a result from discussions between the second author and Simon Schmidt in 2018/2019. It was then conveyed to the third author within his Bachelor's thesis project \cite{zenner} in 2021 and this article contains some of its main results 
(mainly Sections 3.1--3.4 and 4.1). Subsequently, the first author extended the theory in her Master's project \cite{trieb} in  2022 (mainly Sections 3.5--3.7, 4.2, and 5--6). Finally, Bj\"orn Sch\"afer wrote his Master's thesis on nuclearity of hypergraph $C^*$-algebras \cite{schaeferthesis}; the main results can be found in \cite{Schaefer}. All theses were  written under the supervision of the second author, Moritz Weber. The PhD student Nicolas Faro\ss\; of Moritz Weber studied quantum symmetries of hypergraph $C^*$-algebras, see \cite{faross}.   MW has been supported by the SFB-TRR 195, the Heisenberg program of the DFG and a joint OPUS-LAP grant with Adam Skalski.

\section{\textbf{Preliminaries on graph $C^*$-algebras}}
\label{sec:headings}

\subsection{Universal C*-algebras}
Before speaking about graph $C^*$-algebras we need to introduce some notations on universal $C^*$-algebras. See \cite{isem24} for universal $C^*$-algebras and \cite{blackadar_operator, C_algebras_by_example} for the general theory of $C^*$-algebras. 

Let $E=\{x_i \ \vert \ i\in I\}$ be a set of elements indexed by some index set $I$. From these we can construct noncommutative monomials and noncommutative polynomials, by concatenation of letters from $E$. By adding another set $E^*=\{x^*_i \ \vert \ i\in I\}$ which is disjoint from $E$ and by defining an involution on $E\cup E^*$ we obtain the \emph{free *-algebra} $P(E)$ on the generator set $E$. We can view relations as a subset of polynomials $R\subset P(E)$. By taking the two sided *-ideal $J(R)\subset P(E)$ generated by $R$ we define the \emph{universal} *-\emph{algebra} $A(E/R):=P(E)/J(R)$ as a quotient space. Recall that a $C^*$-seminorm on a *-algebra $A$ is given by a map $p:A\rightarrow [0,\infty)$, such that 
\begin{enumerate} [(a)]
    \item $p(\lambda x)=\vert\lambda\vert p(x)$ and $p(x+y)\leq p(x)+p(y)$ for all $x,y\in A$ and $\lambda\in\mathbb{C}$,
    \item  $p(xy)\leq p(x)p(y)$ for all $x,y\in A$,
    \item  and $p(x^*x)=p(x)^2$ for all $x\in A$ 
\end{enumerate}
holds. We put
\begin{align*}
    \Vert x\Vert:=\sup\{p(x) \ \vert \ p \ \text{is a $C^*$-seminorm on} \ A(E\vert R) \} .
\end{align*}
and if $\Vert x\Vert<\infty$ for all $x\in A(E\vert R)$, then $\Vert\cdot\Vert$ is a $C^*$-seminorm. In that case, we define the universal $C^*$-algebra $C^*(E\vert R)$ as the completion with respect to the norm $\Vert\overset{\cdot}{x}\Vert:=\Vert x\Vert$, where $\overset{\cdot}{x}\in A(E\vert R)/\{x\in A(E\vert R) \ \vert \ \Vert x\Vert=0\}$ is the equivalence class of $x$.
Hence, we have
\begin{align*}
    C^*(E\vert R):=\overline{A(E\vert R)/\{x\in A(E\vert R) \ \vert \ \Vert x\Vert=0\}}^{\Vert\cdot\Vert}.
\end{align*}
This is a $C^*$-algebra, see for instance \cite{isem24}. By abuse of notation we write for the elements in $C^*(E\vert R)$ just $x\in C^*(E\vert R)$.
The following lemma  provides a useful tool for proving that a universal $C^*$-algebra actually exists, see \cite{isem24}.

\begin{lemma} 
\label{existenz}
Let $E=\{x_i \ \vert \ i\in I\}$ be a set of generators and $R\subset P(E)$ be relations.
If a constant $C$ exists such that $p(x_i)<C$ holds for all $i\in I$ and all $C^*$-seminorms $p$ on $A(E\vert R)$, then  $\Vert x\Vert<\infty$ holds for all $x\in A(E\vert R)$. In that case, we say that the universal $C^*$-algebra exists.
\end{lemma}

\begin{lemma}\label{existenz2}
If a universal $C^*$-algebra is generated only by projections and partial isometries, it exists.
\end{lemma}
\begin{proof}
  If $x$ is a projection, i.e. $x=x^*=x^2$, then $p(x)^2=p(x^*x)=p(x)$ and hence $p(x)\leq 1$ for all $C^*$-seminorms $p$. If $x$ is a partial isometry, i.e. $x^*x$ is a projection, then $p(x)^2=p(x^*x)\leq 1$. Hence, we may choose $C=1$ in the previous lemma. 
\end{proof}
Note that a universal $C^*$-algebra may exist, but it could still be the case that it is trivial (i.e. zero). 
Universal $C^*$-algebras satisfy the so called universal property; this is a tool for obtaining *-homomorphisms between $C^*$-algebras: Let $E=\{x_i \ \vert \ i\in I\}$ be a set of generators and $R\subset P(E)$ be relations. We say that elements $\{y_i \ \vert \ i\in I\}$ in some *-algebra $B$ satisfy the relations $R$, if every polynomial $p\in R$ vanishes, when we replace $x_i$ with $y_i$. In that case,  there exists a unique *-homomorphism $\phi:C^*(E\vert R)\to B$, sending  $x_i$ to $y_i$ for all $i\in I$.

\subsection{Graph C*-algebras}
In this section we recall the definition of graph $C^*$-algebras. For more details see \cite{Raeburn}. 

\begin{definition}
    A \emph{directed graph} $\Gamma=(V,E,r,s)$ consists of two countable sets $V$, $E$ and functions $r:E\rightarrow V$, $s:E\rightarrow V$. The elements of $V$ are called \emph{vertices} and the elements of $E$ are called \emph{edges}. The map $r$ is named \emph{range} map and the map $s$ is named \emph{source} map. 
    We say that $v\in V$ is a \emph{sink} iff the set $s^{-1}(v)$ is empty and we call $v$ a \emph{source} iff $r^{-1}(v)$ is empty.
\end{definition}

Throughout this paper we always restrict to finite directed graphs unless specified otherwise.

\begin{definition}
 \label{graphalgebra}
 Let $\Gamma=(V,E,r,s)$ be a (finite) graph. The \emph{graph} 
 $C^*$-\emph{algebra} $C^*(\Gamma)$ of the graph $\Gamma$ is the universal $C^*$-algebra generated by mutually orthogonal projections $p_v$ for all $v\in V$ and partial isometries $s_e$ for all $e\in E$ such that the following relations hold:
 \begin{itemize}
     \item [(CK1)] $s_e^*s_f=\delta_{ef}p_{r(e)}$ for all $e,f\in E$. 
     \item [(CK2)] $p_v=\sum_{\substack{e\in E\\s(e)=v}}s_es_e^*$ for all $v\in V$, in case that $v$ is no sink.
 \end{itemize}
 The relations are called \emph{Cuntz-Krieger relations}\index{Cuntz-Krieger relations}. Elements $\{S_e, P_v\;|\; e \in E, v \in V\}$ in a $C^*$-algebra $A$ fulfilling the relations are called \emph{Cuntz-Krieger $\Gamma$-family}.
 \end{definition}
 Graph $C^*$-algebras exist (as universal $C^*$-algebras) due to Lemma \ref{existenz2}.

 \begin{remark}\label{sps}
 From Relation (CK1) it immediately follows that  $s_e=s_ep_{r(e)}$ and $s_e=p_{s(e)}s_e$ hold for all $e\in E$. 
 \end{remark}

\section{\textbf{Definition and examples of hypergraph $C^*$-algebras}}
\label{sec:others}
In this section, we associate a $C^*$-algebra to a finite hypergraph; later in this section, we will also give a definition for an infinite hypergraph. We study some first properties and we give some first examples. We establish their relation to graph $C^*$-algebras and to ultragraph $C^*$-algebras.

\subsection{Definition of hypergraph C*-algebras}
The main idea of hypergraphs is to extend the source and range map to power sets, such that edges can connect sets of vertices instead of just single vertices. We define directed hypergraphs as follows, cf also \cite{directedhyper}.

\begin{definition}
    A \emph{directed hypergraph} $H\Gamma=(V,E,r,s)$ consists of two countable sets $V$ and $E$ and two maps $r,s:E\rightarrow \mathcal{P}(V)\backslash\{\emptyset\}$, where $\mathcal P(V)$ denotes the powerset of $V$. The set $V$ contains vertices, while the set $E$ contains \emph{hyperedges}.
    We call a vertex $v \in V$ a \emph{source}\index{Source!Hypergraph} iff $v \notin r(e)$ for all $e \in E$ and we call it a \emph{sink}\index{Sink!Hypergraph} iff $v \notin s(e)$ for all $e \in E$.
\end{definition}

In this article, all hypergraphs are finite, i.e.  the set of vertices and edges are both finite, unless explicitly stated otherwise.

\begin{remark} \label{orderrelation}
    Recall that for two projections $p,q$ the relation $p\leq q$ holds if and only if $pq=p=qp$ holds. The finite sum of projections is a projection if and only if the projections are mutually orthogonal. Moreover, if $p_i\leq q$ for all $i\in I$ and $q\leq \sum p_i$ for a projection $q$ and finitely many mutually orthogonal projections $p_i$, we infer $q=\sum p_i$.
\end{remark}

\begin{definition}\label{Hypergraph_Def}
Let $H\Gamma=(V,E,r,s)$ be a (finite) hypergraph. The \emph{hypergraph $C^*$-algebra} $C^*(H\Gamma)$ is the universal $C^*$-algebra generated by mutually orthogonal projections $p_v$ for all $v\in V$ and partial isometries $s_e$ for all $e\in E$ such that the following relations hold:
\begin{itemize}
    \item [(HR1)] $s_e^*s_f=\delta_{ef}\sum_{v\in r(e)}p_v$ for all $e,f\in E$.
    \item [(HR2a)] $s_es_e^*\leq\sum_{v\in s(e)}p_v$ for all $e\in E$.
    \item [(HR2b)] $p_v \leq\sum_{\substack{e\in E\\v\in s(e)}}s_es_e^*$ for all $v \in V$, where $v$ is no sink.
\end{itemize}
We call these relations the \emph{hypergraph relations}. Elements $\{S_e, P_v\}$ in a $C^*$-algebra $A$ fulfilling the hypergraph relations are called \emph{Cuntz-Krieger $H\Gamma$-family}.
\end{definition}
Hypergraph $C^*$-algebras exist (as universal  $C^*$-algebra) by Lemma \ref{existenz2}.



\subsection{Hypergraph $C^*$-algebras generalize graph $C^*$-algebras}

Note that every graph $\Gamma=(V,E,r,s)$ is also a hypergraph $H\Gamma=(V,E,r',s')$ by defining $r':E\rightarrow\mathcal{P}(V), e\mapsto \{r(e)\}$ and $s':E\rightarrow\mathcal{P}(V),e\mapsto\{s(e)\}$. In other words, graphs are hypergraphs with the restriction $\lvert s(e)\rvert=\lvert r(e)\rvert=1$ for all edges $e\in E$.
We show that every graph $C^*$-algebra is also a hypergraph $C^*$-algebra. 

\begin{proposition}
\label{nontrivial}
Consider a graph $\Gamma=(V,E,r,s)$ and interpret it as a hypergraph $H\Gamma=(V,E,r',s')$ in the sense as described before. For our graph $C^*$-algebra we write 
\begin{align*}
    C^*(\Gamma)=C^*(\tilde{s}_e, \ e\in E ; \tilde{p}_v, \ v\in V \ \vert \ \tilde{p}_v\tilde{p}_w=0, v\not=w;  \tilde{s}_e^*\tilde{s}_f=\delta_{ef}\tilde{p}_{r(e)} ; \sum_{\substack{e\in E \\s(e)=w}}\tilde{s}_e\tilde{s}_e^*=\tilde{p}_w)
\end{align*}
where $\tilde{s}_e$ is a partial isometry for all $e\in E$ and $\tilde{p}_v$ is a projection for all $v\in V$. Then we have $C^*(\Gamma)\cong C^*(H\Gamma)$. 
\begin{proof}
First, we check that the generators of $C^*(\Gamma)$ fulfill the relations of $C^*(H\Gamma)$. Since the only element in the set $r'(e)$ is the vertex $r(e)$ we have
\begin{align*}
    \tilde{s}_e^*\tilde{s}_f=\delta_{ef}\tilde{p}_{r(e)}=\delta_{ef}\sum_{\substack{v\in V\\v\in r'(e)}}\tilde{p}_v.
    \end{align*}
    We see that Relation (HR1) is fulfilled. For the same reasons it follows for $v\in V$ with $s^{-1}(v)\not=\emptyset$ that
    \begin{align*}
 \tilde{p}_v =\sum_{\substack{e\in E\\  v=s(e)}}\tilde{s}_e\tilde{s}_e^* =\sum_{\substack{e\in E\\  v\in s'(e)}}\tilde{s}_e\tilde{s}_e^*.
\end{align*}
Hence, Relation (HR2b) is satisfied, while (HR2a) follows from Remark \ref{sps}.

Conversely,  the generators of $C^*(H\Gamma)$ satisfy the Relations (CK1) and (CK2) of $C^*(\Gamma)$: Using the same argument as in the above direction, we have\begin{align*}
    s_e^*s_f=\delta_{ef}\sum_{\substack{v\in V\\ v\in r'(e)}}p_v=\delta_{ef}p_{r(e)} .
\end{align*}
We see that Relation (CK1) is satisfied. To show that (CK2) is fulfilled we need Relations (HR2a) and (HR2b). Let $v\in V$ with $s^{-1}(v)\not=\emptyset$ and hence there exists at least one $f\in E$ with $s(f)=v$. With (HR2a) it follows
\begin{align*}
    s_es_e^*\leq\sum_{\substack{v\in V \\ v\in s'(e)}} p_v=p_{s(e)}
\end{align*}
and using Relation (HR2b) we have
\begin{align*}
  p_v\leq \sum_{\substack{e\in E\\  v\in s'(e)}} s_es_e^*= \sum_{\substack{e\in E\\  v= s(e)}} s_es_e^*.
\end{align*}
By Remark \ref{orderrelation} we conclude
\begin{align*}
    p_v
    =\sum_{\substack{e\in E\\  v=s(e)}}s_es_e^*.
\end{align*}
By the universal properties of $C^*(\Gamma)$ and $C^*(H\Gamma)$ we find an isomorphism mapping $\tilde s_e$ to $s_e$ and $\tilde p_v$ to $p_v$.
\end{proof}
\end{proposition}

\subsection{Hypergraph $C^*$-algebras are unital}

The next statement is known for graph $C^*$-algebras (associated to finite graphs), but it is still true for the generalization to hypergraph $C^*$-algebras. Recall that all our hypergraphs are finite.

\begin{proposition}
\label{unit}
For every hypergraph $H\Gamma=(V,E,r,s)$ and hypergraph C*-algebra $C^*(H\Gamma)$ we have that $\sum_{v\in V} p_v$ is the unit element in $C^*(H\Gamma)$ and therefore, $\sum_{v\in V} p_v=1$.

\begin{proof} 
Using Relation (HR1), we have 
\begin{align*}
    s_e\sum_{v\in V}p_v&=s_es_e^*s_e\sum_{v\in V}p_v\\
    &=s_e\sum_{v\in r(e)}p_v\sum_{v\in V}p_v\\
    &=s_e\sum_{v\in r(e)}p_v\\
    &=s_e  .
    \end{align*}
    Using the same trick as in the argument before but with Relation (HR2a), we also have $(\sum_{v\in  V}p_v)s_e=s_e$.
    Notice that we have $\sum_{v\in V}p_vp_w=p_w=p_w\sum_{v\in V}p_v$ for all $w\in V$ and $(\sum_{v\in V}p_v)^2=\sum_{v\in V}p_v=(\sum_{v\in V}p_v)^*$.
    We conclude that $\sum_{v\in V}p_v$ is the unit element in $C^*(H\Gamma)$.
    \end{proof}
\end{proposition}

Next, we show that the analogous of Remark \ref{sps} holds true for hypergraph $C^*$-algebras. 

\begin{proposition}\label{source_range_projections}
    Let $H\Gamma=(V,E, r, s)$ be a  hypergraph. For each Cuntz-Krieger $H\Gamma$-family $\{p_v, s_e \}$ it holds 
    \begin{equation*} \label{generalrelations}
        \left(\sum_{v\in s(e)} p_v\right) s_e= s_e = s_e\left(\sum_{v\in r(e)} p_v\right).
    \end{equation*}
    In particular, $s_es_f^*=0$, if $r(e)\cap r(f)=\emptyset$.
\end{proposition}

\begin{proof}
By (HR2a), we have 
    \begin{equation*}
        \left(\sum_{v\in s(e)} p_v\right)s_e= \left(\sum_{v\in s(e)} p_v\right)s_es_e^*s_e=s_es_e^*s_e=s_e
    \end{equation*}
and by (HR1), we have
    \begin{equation*}
        s_e\left(\sum_{v\in r(e)} p_v\right)=s_es_e^*s_e=s_e.
    \end{equation*}
\end{proof}

\subsection{First examples} 

In this section, we list a number of first examples. Further, more interesting ones may be found in Section 4.
We can express the Toeplitz algebra as a hypergraph $C^*$-algebra. We can even consider two different hypergraphs as we see in the following example.

\begin{proposition}\label{toeplitz}
Consider the hypergraphs $H\Gamma_1$ given by $V_1=\{v,w\}$, $E_1=\{e\}$ and $s(e)=\{w\}$, $r(e)=\{v,w\}$, as well as $H\Gamma_2$ given by $V_2=\{v,w\}$, $E_2=\{e\}$ and $s(e)=\{v,w\}$, $r(e)=\{w\}$, depicted in Fig. \ref{fig:my_label}.
Both associated $C^*$-algebras are isomorphic to the Toeplitz algebra $\mathcal{T}$.

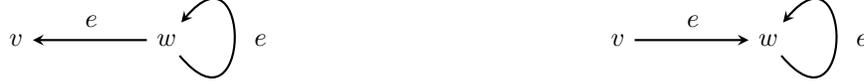
\begin{figure} [!ht]
    \centering
        \begin{center}
        \begin{minipage}{\linewidth}
        \centering
            \begin{tikzpicture}[ > = stealth, auto, thick]
            \node (v1) at (0,0) {$v$};
            \node (v2) at (2,0) {$w$};
            \path[->] (v2) edge[] (v1);
            \node (e1) at (1,0.25) {$e$};
            \path[->,in=50,out=-50,loop,scale=3] (v2) edge (v2);
            \node (e3) at (3.25,0) {$e$};
        
            \node (v1) at (8,0) {$v$};
            \node (v2) at (10,0) {$w$};
            \path[->] (v1) edge[] (v2);
            \node (e1) at (9,0.25) {$e$};
            \path[->,in=50,out=-50,loop,scale=3] (v2) edge (v2);
            \node (e3) at (11.25,0) {$e$};
            \end{tikzpicture}
        \end{minipage}
    \end{center}
    \caption{Hypergraphs generating the Toeplitz algebra.}
    \label{fig:my_label}
\end{figure}
    
\end{proposition}

\begin{proof}
    For the first hypergraph, $H\Gamma_1$, we get by the hypergraph relations
        \begin{align*}
            s_e^*s_e=p_v+p_w, \qquad s_es_e^*=p_w \textnormal{ (as }s_es_e^*\leq p_w\leq s_es_e^*).
        \end{align*}
    Since $p_v+p_w=1$ by Lemma \ref{unit}, $s_e$ is an isometry. By the universal property of the Toeplitz algebra, which we view as the universal $C^*$-algebra generated by an isometry $u$, we get a *-homomorphism $\phi:\mathcal{T}\rightarrow C^*(H\Gamma_1)$ defined by 
        \begin{align*}
            u&\mapsto s_e, \quad 
            1 \mapsto p_v+p_w=s_e^*s_e.
        \end{align*}
    On the other hand we can define a Cuntz-Krieger $H\Gamma_1$-family in $\mathcal{T}$ as follows
        \begin{align*}
            S_e:=u,\qquad
            P_w:=uu^*,\qquad
            P_v:=1-uu^*.
        \end{align*}
    The universal property then yields a *-homomorphism $\psi:C^*(H\Gamma_1)\rightarrow \mathcal{T}$ defined by 
        \begin{align*}
            s_e\mapsto S_e, \quad
            p_w \mapsto P_w,\quad
            p_v \mapsto P_v.
        \end{align*}
        We conclude that $\phi \circ \psi=\id_{C^*(H\Gamma_1)}$ and $\psi \circ \phi =\id_{\mathcal{T}}$. Thus we get the required isomorphism.
    
    The proof for the second hypergraph follows similarly, using that $s_e^*$ is an isometry as  $s_es_e^*=p_v+p_w=1$.
\end{proof}

Another basic example of a hypergraph $C^*$-algebra is the Cuntz algebra $\mathcal O_n$ \cite{Cuntz}.

\begin{proposition}
    \label{hypercuntz}
Let $n\in\mathbb{N}$ and $n\geq 2$. Consider the hypergraph $H\Gamma$ with vertices $V=\{v_1,...,v_n\}$, edges $E=\{e_1,...,e_n\}$ and $r(e_i)=\{v_1,...,v_n\}, \ s(e_i)=\{v_i\}$ for all $i=1,...,n$. Then $C^*(H\Gamma)\cong \mathcal{O}_n$.
\end{proposition}
\begin{proof}
    By the relations of the hypergraph $C^*$-algebra $C^*(H\Gamma)$ we have
\begin{align*}
     &s_{e_i}^*s_{e_j}=\delta_{ij}\sum_{j=1}^n p_{v_j}=1 \ \text{for all} \ i,j=1,...,n, \text{ using Lemma \ref{unit}},\\
     &\text{and } s_{e_i}s_{e_i}^*\leq p_{v_i}\leq s_{e_i}s_{e_i}^*.\\
     \end{align*}
     Hence, the isometries $s_{e_i}$ satisfy $\sum_{i=1}^ns_{e_i}s_{e_i}^*=\sum_{i=1}^n p_{v_i}=1$
     and we obtain $\phi:\mathcal O_n\to C^*(H\Gamma)$ by the universal property.
     We obtain an inverse *-homomorphism, since the generators $S_{e_i}:=S_i\in\mathcal O_n$ and the projections $P_v:=S_iS_i^*$ satisfy the relations of $C^*(H\Gamma)$. 
\end{proof}

\subsection{Paths in hypergraph $C^*$-algebras}

Graph $C^*$-algebras are spanned by words in paths. This cannot be generalized directly, as paths in hypergraphs are of different quality -- it is a priori  unclear how to define them and we give various definitions.

\begin{definition}
    Let $\mu=(\mu_1...\mu_n)$ be a tuple of edges in $H\Gamma$. Then we call $\mu$
    \begin{enumerate}
        \item a \emph{perfect path}\index{Path!Perfect}, if $s(\mu_{j+1})=r(\mu_j)$ for all $j\in \{1,...,n\}$;
        \item a \emph{quasi perfect path}\index{Path!Quasi perfect}, if $s(\mu_{j+1})\subseteq r(\mu_j)$ for all $j\in \{1,...,n\}$;
        \item a \emph{partial path}\index{Path!Partial}, if $s(\mu_{j+1})\cap r(\mu_j)\neq \emptyset$ for all $j\in \{1,...,n\}$.
    \end{enumerate}
     We define $s_{\mu}:=s_{\mu_1}...s_{\mu_n}$ and $s_v:=p_v$.
\end{definition}

Restricting to a particular nice class of paths, we may deduce a result in analogy to graph $C^*$-algebras.

\begin{proposition}\label{span_hypergraph_algebra}
    Let $H\Gamma=(V, E, r, s)$ be a  hypergraph with only perfect and quasi perfect paths and a Cuntz-Krieger $H\Gamma$-family $\{p_v, s_e\}$. It holds: 
    \begin{equation*}
        C^*(H\Gamma)=\overline{span}\left\{s_{\mu}s_{\nu}^* \;|\; 
        \mu=(\mu_1...\mu_k), \nu=(\nu_1...\nu_m), \mu_i,\nu_j\in V\cup E \text{ for all } i,j, \text{ and } r(\mu)\cap r(\nu)\neq \emptyset\right\}.
    \end{equation*}
\end{proposition}

\begin{proof}
Let $\mu=(\mu_1,\mu_2)$ be a quasi perfect path. Then 
$\left(\sum_{v\in r(\mu_1)}p_v\right)\geq\left(\sum_{v\in s(\mu_2)}p_v\right)$ and we have
\[\left(\sum_{v\in r(\mu_1)}p_v\right)s_{\mu_2}=
\left(\sum_{v\in r(\mu_1)}p_v\right)\left(\sum_{v\in s(\mu_2)}p_v\right)s_{\mu_2}=s_{\mu_2}
\]
by Proposition \ref{source_range_projections}. Hence, if $\nu=(\nu_1,\nu_2)$ is another quasi perfect path, then 
\[s_{\nu}^*s_\mu=s_{\nu_2}^*s_{\nu_1}^*s_{\mu_1}s_{\mu_2}=\delta_{\nu_1\mu_1}s_{\nu_2}^*\left(\sum_{v\in r(\mu_1)}p_v\right)s_{\mu_2}=\delta_{\nu_1\mu_1}s_{\nu_2}^*s_{\mu_2}=\delta_{\nu_1\mu_1}\delta_{\nu_2\mu_2}\left(\sum_{v\in r(\mu_1)}p_v\right).\]
An iteration  yields for quasi perfect paths $\mu, \nu, \alpha, \beta$:  
        \begin{align*}
            (s_\mu s_\nu^*)(s_\alpha s_\beta^*)=\begin{cases} s_{\mu\alpha'}s_\beta^* & \alpha = \nu\alpha'\\
            s_\mu s_{ \beta\nu'}^* & \nu=\alpha \nu'\\
            \sum_{v\in I}s_{\mu\nu} s_\beta^* & \nu =\alpha, \text{for some index set $I$} \\
            0 & \text{else}.\end{cases}
        \end{align*}
\end{proof}

\begin{remark}
    For a (quasi) perfect path $\mu = (\mu_1,\ldots, \mu_n)$ the element $s_\mu$ is always a partial isometry. This is not the case if we deal with partial paths: The projections onto vertices then no longer disappear in chains of partial isometries. 
\end{remark}

\subsection{(Infinite) hypergraph $C^*$-algebras generalize ultragraph C*-algebras}

Ultragraph C*-algebras, first defined by \cite{Tomforde}, form a subclass of hypergraph C*-algebras, as we will show in this subsection. The main intention of Tomforde's construction was to find a unified approach to graph $C^*$-algebras and Exel-Laca algebras generalizing Cuntz-Krieger algebras to the infinite setting. 
We adapt the definition in order to define hypergraph $C^*$-algebras for infinite hypergraphs.

\begin{definition}
    A \emph{(directed) ultragraph} $U\Gamma=(V,E,r,s)$ is defined by a countable set of vertices $V$  and a countable set $E$ of edges together with a source map $s:E\to V$ and a range map $r:E\to\mathcal P(V)\backslash\{\emptyset\}$.
\end{definition}

Hence, an ultragraph is a (not necessarily finite) hypergraph with the condition $\lvert s(e)\rvert=1$ for all $e\in E$.
To clarify the differences between graphs, ultragraphs and hypergraphs, we illustrate the possible edges in graphs, ultragraphs and hypergraphs in Figure \ref{fig:three graphs}.

\begin{figure}[!ht]
     \centering
     \begin{subfigure}[b]{0.3\textwidth}
         \centering
            \begin{center}
        \begin{minipage}{\linewidth}
        \centering
            \begin{tikzpicture}[ > = stealth, auto, thick, scale=0.8]
                \node (v0) at (0,0) {};
                \fill (v0) circle[radius=2pt];
                \node (v1) at (2,0) {};
                \fill (v1) circle[radius=2pt];
                \path[->] (v0) edge (v1); 
                \node (e0) at (1,0.25) {$e$};
                \node (v9) at (2,0.75) {};
                \node (v9) at (2,-0.75) {};
            \end{tikzpicture}
        \end{minipage}
    \end{center}
         \caption{Graph ($\lvert s(e)\rvert=\lvert r(e)\rvert=1$)}
     \end{subfigure}
     \hfill
     \begin{subfigure}[b]{0.3\textwidth}
         \centering
            \begin{center}
        \begin{minipage}{\linewidth}
        \centering
            \begin{tikzpicture}[ > = stealth, auto, thick, scale=0.8]
                
                \node (v4) at (0,0) {};
               \fill (v4) circle[radius=2pt];
               \node (v5) at (2,0.75) {};
                \fill (v5) circle[radius=2pt];
                \node (v6) at (2,-0.75) {};
                \fill (v6) circle[radius=2pt];
                \path[->] (v4) edge (v5); 
                \node (e0) at (1,0.6) {$e$};
                \path[->] (v4) edge (v6); 
                \node (e0) at (1,-0.6) {$e$};
            \end{tikzpicture}
        \end{minipage}
    \end{center}
         \caption{Ultragraph ($\lvert s(e)\rvert=1$)}
    \end{subfigure}
     \hfill
     \begin{subfigure}[b]{0.3\textwidth}
         \centering
            \begin{center}
        \begin{minipage}{\linewidth}
        \centering
            \begin{tikzpicture}[ > = stealth, auto, thick, scale=0.8]
                \node (v8) at (0,0) {};
                \fill (v8) circle[radius=2pt];
                \node (v9) at (2,0.75) {};
                \fill (v9) circle[radius=2pt];
                \node (v10) at (2,-0.75) {};
                \fill (v10) circle[radius=2pt];
                \path[->] (v8) edge (v9); 
                \node (e0) at (1,0.6) {$e$};
                \path[->] (v8) edge (v10); 
                \node (e0) at (1,-0.6) {$e$};
                \path[->] (v9) edge (v10); 
                \node (e0) at (2.25,0) {$e$};
                \path[->, in=0, out=90, loop] (v9) edge (v9); 
                \node (e0) at (2.6,1.2) {$e$};
            \end{tikzpicture}
        \end{minipage}
    \end{center}
         \caption{Hypergraph}
     \end{subfigure}
        \caption{An edge $e$ in a graph, an ultragraph and a hypergraph.}
        \label{fig:three graphs}
\end{figure}
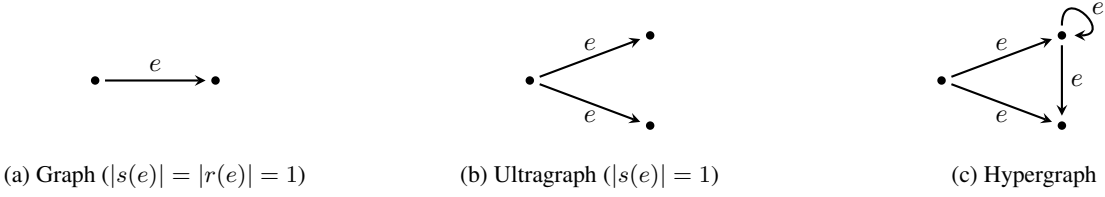


\begin{definition} \label{generalized_vertices}
    Let $H\Gamma=(V,E,r,s)$ be a hypergraph. Let 
        \begin{align*}
            \mathcal{V}_0:=\{\{x\}\mid x\in V\}\cup \{s(e),\, r(e) \;|\; e \in E\}\subset \mathcal P(V)
        \end{align*}
        and let $\mathcal V$ be the smallest subcollection of $\mathcal{P}(V)$ containing $\mathcal{V}_0$ which is closed under finite unions and finite intersections. We call the sets $A \in \mathcal{V}$ \emph{generalized vertices}\index{Generalized vertices}. 
\end{definition}

With this notation we can give a slightly different definition of a Cuntz-Krieger family involving the generalized vertices. This definition is adapted from \cite[Def. 2.7]{Tomforde}. 

\begin{definition}\label{Hypergraph_Def_generalized}
    Let $H\Gamma=(V, E, r, s)$ be a (not necessarily finite) hypergraph. A \emph{generalized Cuntz-Krieger $H\Gamma$-family}\index{Cuntz-Krieger family!Generalized} is a collection of partial isometries $\{s_e\;|\; e \in E\}$ and orthogonal projections $\{p_A \; |\; A \in \mathcal{V}\}$ such that 
        \begin{enumerate}
            \item[(GR0)] $p_\emptyset=0$, $p_Ap_B=p_{A \cap B}$ and $p_{A \cup B}=p_A+p_B-p_{A \cap B}$ for all $ A,B \in \mathcal{V}$;
            \item[(GR1)] $s_e^*s_f=\delta_{ef}p_{r(e)}$ for all $e,f \in E$;
            \item[(GR2a)] $s_es_e^*\leq p_{s(e)}$ for all $e \in E$;
            \item[(GR2b)] $p_{\{v\}} \leq \sum_{e \in E, v \in s(e)}s_es_e^*$ for all $v \in V$ which emit at least one and at most finitely many edges (i.e. $\emptyset\neq s(e)$ and $\lvert s(e)\rvert<\infty$).
        \end{enumerate}
The universal $C^*$-algebra $C^*(H\Gamma)$ generated by these generators and relations is called a \emph{hypergraph $C^*$-algebra}.
\end{definition}

    For Definition \ref{Hypergraph_Def_generalized}, the clue is that this definition also makes sense for infinite hypergraphs. In the infinite case, we have to ensure that all sums of projections are well defined. Thus, all infinite sums must be avoided. Relation (GR2b) is restricted accordingly.
     But for edges with infinitely many vertices in their source or range, we have to adjust the other hypergraph relations as well, due to the possibly infinite sums
        \begin{align*}
            \sum_{v \in r(e)}p_v \qquad \text{and} \qquad \sum_{v \in s(e)}p_v.
        \end{align*}
    The approach involving generalized vertices forces the existence of the required projections in a natural way. And it turns out that in the finite case, both definitions still coincide.

\begin{lemma}
    Let $H\Gamma$ be a finite hypergraph. Then the Cuntz-Krieger families in Definition  \ref{Hypergraph_Def} and Definition \ref{Hypergraph_Def_generalized} generate the same $C^*$-algebra. 
\end{lemma}

\begin{proof}
Given a Cuntz-Krieger family as in Def. \ref{Hypergraph_Def} of projections $p_v, v\in V$ and partial isometries $s_e$, $e\in E$, the sums $p_A:=\sum_{v \in A}p_v$ are projections for all generalized vertices $A \in \mathcal{V}$ and we obtain a generalized Cuntz-Krieger $H\Gamma$-family as in Def. \ref{Hypergraph_Def_generalized}. Conversely,  in a generalized Cuntz-Krieger $H\Gamma$-family as in Def. \ref{Hypergraph_Def_generalized}, the projections $p_v:=p_{\{v\}}$, $v\in V$ and the partial isometries $s_e$, $e\in E$ form a Cuntz-Krieger $H\Gamma$-family as in Def. \ref{Hypergraph_Def}.
\end{proof}

If $H\Gamma$ is an ultragraph, then $C^*(H\Gamma)$ is an ultragraph $C^*$-algebra in Tomforde's sense \cite[Def. 2.7]{Tomforde} and vice versa.
This class is nuclear, see \cite[Thm. 5.22]{Morita_ultragraphs}. 

\begin{proposition}\label{MoritaUltra}
    Any ultragraph $C^*$-algebra is Morita equivalent to a graph $C^*$-algebra. In particular, all ultragraph $C^*$-algebras are nuclear.
\end{proposition}

\begin{remark}\label{unit_infinite_hypergraph}
    In Lemma \ref{unit} we showed, that hypergraph $C^*$-algebras are always unital, if the underlying hypergraph is finite. This fails to be true for an infinite number of vertices, as an infinite sum of projections cannot converge in norm. 
    However, as in \cite[Lem. 3.2]{Tomforde}, one can show  that $C^*(H\Gamma)$ is unital if and only if 
        \begin{align*}
            V \subset \left\{\bigcup_{i=1}^n\left(\bigcap_{e\in X_i}r(e)\right) \cup \bigcup_{i=1}^m\left(\bigcap_{e\in Y_i}s(e)\right)\cup F\; \Big|\; X_i, Y_i\; \subseteq E \; \text{finite}, \; F\subseteq V \;\text{finite}\right\}.
        \end{align*}
    The proof of this follows in exactly the same way as in \cite[Lem. 3.2]{Tomforde}, except for the construction of the approximate unit. 
\end{remark}

\subsection{Decomposition of ranges}

Multiple hypergraphs can have a similar corresponding $C^*$-algebra. By decomposing the range of an edge, we give a concrete way to construct new hypergraphs while leaving the corresponding $C^*$-algebra invariant. Also, this  gives us some information about the relation between graph and ultragraph $C^*$-algebras and it shows the crucial differences to hypergraphs. The idea of the proof is based on \cite{Morita_ultragraphs}, where it is shown that each (infinite) ultragraph $C^*$-algebra is Morita equivalent to a graph algebra. 

\begin{proposition}\label{rangesingularization}\index{Decomposition of ranges}
    Let $H\Gamma=(V,E,r,s)$ be a finite hypergraph. Define the hypergraph $\tilde{H\Gamma}=(\tilde V,\tilde E,\tilde{r},\tilde{s})$ as
        \begin{align*}
            &\tilde V:=V,\\
            &\tilde E:=\{(e,v)\; |\; e \in E, \; v \in r(e)\},\\
            &\tilde{r}((e,v)):=v,\\
            &\tilde{s}((e,v)):=s(e).
        \end{align*}
    The corresponding hypergraph $C^*$-algebras are isomorphic, i.e. $C^*(H\Gamma)\cong C^*(\tilde{H\Gamma})$. In particular it holds that $\tilde{r}:\tilde E\rightarrow \tilde V$ (rather than $\tilde r:\tilde E\to\mathcal P(\tilde V)\backslash \{\emptyset\})$, i.e. we have $\lvert \tilde r(\tilde e)\rvert =1$ for all $\tilde e\in\tilde E$. 
\end{proposition}

\begin{proof}
    Let $\{q_v\;|\; v \in \tilde V\}$, $\{t_\alpha \;|\; \alpha \in \tilde E\}$ be the universal Cuntz-Krieger $\tilde{H\Gamma}$-family. We define
        \begin{align*}
            P_v&:=q_v  \hspace{68pt}\forall v \in V,\\
            S_e&:=\sum_{v \in r(e)}t_{(e,v)} \hspace{30pt} \forall e \in E.
        \end{align*}
    The elements $\{P_v\;|\; v \in V\}$ are mutually orthogonal projections and a quick calculation shows that $\{S_e \;|\; e \in E\}$ are partial isometries with mutually orthogonal ranges. Together they form a Cuntz-Krieger $H\Gamma$-family in $C^*(\tilde{H\Gamma})$, as we see in the following. \\
    
    For (HR1) of $C^*(H\Gamma)$, check that (HR1) of $C^*(\tilde{H\Gamma})$ implies
    \[  S_e^*S_f=\sum_{v \in r(e)}\sum_{w \in r(f)}t_{(e,v)}^*t_{(f,w)}=\delta_{ef}\sum_{v \in r(e)}q_v=\delta_{ef}\sum_{v \in r(e)}P_v.\]

     For (HR2a) of $C^*(H\Gamma)$, using that the ranges of $(e,w)$ and $(e,z)$ for distinct vertices $w$, $z$ are disjoint we get using Proposition \ref{source_range_projections} that $t_{(e,w)}t_{(e,z)}^*=0$ for $w \neq z$. Since Relation (HR2a) of $C^*(\tilde{H\Gamma})$ implies that $t_{(e,w)}t_{(e,w)}^*\leq \sum_{v \in s(e)}q_v$ for all $w \in r(e)$ we get
     \[
        S_eS_e^*=\sum_{w \in r(e)}t_{(e,w)}\sum_{z \in r(e)}t_{(e,z)}^*
        =\sum_{w \in r(e)}t_{(e,w)}t_{(e,w)}^*
        \leq \sum_{v \in s(e)}q_v
        =\sum_{v \in s(e)}P_v.\]
        
    For (HR2b) of $C^*(H\Gamma)$, using (HR2b) of $C^*(\tilde{H\Gamma})$ and the orthogonality of the ranges of $(e,w)$ and $(e,z)$ for distinct vertices $w$ and $z$, we get
    \begin{align*}
        P_v&=q_v\\
        &\leq \sum_{\alpha \in \tilde{E}, v \in s(\alpha)}t_\alpha t_\alpha^*\\
        &=\sum_{e \in E, v \in s(e)}\;\sum_{w \in r(e)}t_{(e,w)} t_{(e,w)}^*\\
        &=\sum_{e \in E, v \in s(e)}\;\sum_{w \in r(e)}t_{(e,w)}\sum_{z \in r(e)}t_{(e,z)}^*\\
        &=\sum_{e \in E, v \in s(e)}S_eS_e^*.
    \end{align*}
    
    Hence all hypergraph relations are fulfilled and we thus get a *-homomorphism $\phi:C^*(H\Gamma)\rightarrow C^*(\tilde{H\Gamma})$ which maps the canonical generators $s_e \mapsto S_e$ and $p_v \mapsto P_v$.
    
    To construct the inverse map we define the elements
    \begin{align*}
        Q_v&:=p_v \hspace{40pt} \forall v \in \tilde V,\\
        T_{(e,v)}&:=s_ep_v \hspace{30pt} \forall (e,v) \in \tilde{E}.
    \end{align*}
    Clearly, the elements $Q_v$ are mutually orthogonal projections and $T_{(e,v)}$ is a partial isometry for each $(e,v) \in \tilde{E}$, since $s_ep_vs_e^*$ is a projection, using (HR1). We check that these elements form a Cuntz-Krieger $\tilde{H\Gamma}$-family in $C^*(H\Gamma)$.
    
    For (HR1) of $C^*(\tilde{H\Gamma})$, relation (HR1) of  $C^*(H\Gamma)$ yields
\begin{align*}
            T_{(e,v)}^*T_{(f,w)}&=p_vs_e^*s_fp_w\\
        &=\delta_{ef}  p_v\big(\sum_{z\in r(e)}p_{z}\big)p_w\\
        &=\delta_{ef}\delta_{vw}p_v\\
        &=\delta_{(e,v),(f,w)}Q_v.
\end{align*}

        For (HR2a) of $C^*(\tilde{H\Gamma})$, using the definition of partial isometries and the order relation of projections we get by applying (HR2a) of $C^*(H\Gamma)$:
        \begin{align*}
        T_{(e,v)}T_{(e,v)}^*&=s_ep_vp_vs_e^*\\
        &\leq s_es_e^*\\
        &\leq \sum_{w \in s(e)}p_w\\
        &=\sum_{w \in s((e,v))}Q_w.
        \end{align*}
        
   For (HR2b) of $C^*(\tilde{H\Gamma})$, we get by (HR2b) and (HR1) of $C^*(H\Gamma)$:
   \begin{align*}
        Q_v&=p_v\\
        &\leq \sum_{e \in E, v \in s(e)}s_es_e^*\\
        &\leq \sum_{e \in E, v \in s(e)}\sum_{v \in r(e)}s_ep_vs_e^*\\
        &= \sum_{e \in E, v \in s(e)}\sum_{v \in r(e)}T_{(e,v)}T_{(e,v)}^*\\
        &=\sum_{\alpha \in \tilde{E}, v \in s(\alpha)}T_\alpha T_\alpha^*.
    \end{align*}
    
    The universal property then yields a  *-homomorphism $\psi:C^*(\tilde{H\Gamma})\rightarrow C^*(H\Gamma)$ which is inverse to $\phi$.
\end{proof}

    Instead of a complete decomposition of the range into its single vertices we could have also disassembled it into a disjoint union of nonempty sets, i.e. $r(e)=\mathcal{E}_1\cup \dots \cup \mathcal{E}_n$ and associate to each set $\mathcal{E}_j$ the edge $(e,\mathcal{E}_j)$.

\begin{proposition}
        Let $H\Gamma=(V,E,r,s)$ be a finite hypergraph. For each $e \in E$, let $r(e)=\mathcal{E}_1\cup \dots \cup \mathcal{E}_{n_e}$ for nonempty disjoint sets $\mathcal{E}_j$ and $n_e \in \N$. Define the hypergraph $\tilde{H\Gamma}=(\tilde V,\tilde E,\tilde{r},\tilde{s})$ as
        \begin{align*}
            &\tilde V:=V,\\
            &\tilde E:=\{(e,\mathcal{E}_j)\; |\; e \in E, \; j=1,\dots, n_e\},\\
            &\tilde{r}((e,\mathcal{E}_j)):=\mathcal{E}_j,\\
            &\tilde{s}((e,\mathcal{E}_j))):=s(e).
        \end{align*}
    The corresponding hypergraph algebras are isomorphic, i.e. $C^*(H\Gamma)\cong C^*(\tilde{H\Gamma})$.
\end{proposition}
\begin{proof}
    Similar to the proof of Proposition \ref{rangesingularization}.
\end{proof}

\begin{remark}
    We only get the decomposition for ranges -- the same approach for sources is not possible. For example the element $p_vs_e$ for a vertex $v \in s(e)$ is in general no partial isometry. Also, if it was possible to decompose  ranges \emph{and} sources, we would obtain a graph  $\tilde \Gamma$ such that $C^*(H\Gamma)\cong C^*(\tilde\Gamma)$  given any hypergraph $\Gamma$ contradicting the existence of non-nuclear hypergraph $C^*$-algebras (see  the next section).
\end{remark}

\begin{korollar}
    The $C^*$-algebra of a finite ultragraph is isomorphic to a graph $C^*$-algebra. 
\end{korollar}
\begin{proof}
    Use the decomposition of ranges in order to obtain a graph $\tilde \Gamma$ with $C^*(H\Gamma)\cong C^*(\tilde\Gamma)$.
\end{proof}

We can reverse the above construction and merge edges with similar sources and disjoint ranges. We state the corollary in case of two edges, but it can  directly be generalized to any finite number of edges by iteration.

\begin{korollar}\label{merge_ranges}
    Let $H\Gamma=(V,E,r,s)$ be a finite hypergraph. Consider $e,f \in E$ with $s(e)=s(f)$ and $r(e)\cap r(f)=\emptyset$. The hypergraph $\tilde{H\Gamma}$ given by
        \begin{align*}
            &\tilde V:=V,\\
            &\tilde{E}:=(E\setminus\{e,f\})\cup\{g\},\\
            & \tilde{s}(h):=s(h) \quad \forall h \in E\setminus\{e,f\}, \qquad  \tilde{s}(g):=s(e),\\
            & \tilde{r}(h):=r(h) \quad \forall h \in E\setminus\{e,f\}, \qquad  \tilde{r}(g):=r(e)\cup r(f)
        \end{align*}
    generates an isomorphic hypergraph $C^*$-algebra.
\end{korollar}

\begin{proof}
    We apply Proposition \ref{rangesingularization} to $H\Gamma$ and $\tilde{H\Gamma}$. Both yield the same hypergraph, which gives the required isomorphism.     
\end{proof}

\section{\textbf{Non-nuclear hypergraph $C^*$-algebras}}
\label{sec:nonnuclear}

In this section, we  show that the class of hypergraph $C^*$-algebras is strictly larger than the class of graph $C^*$-algebras. All graph $C^*$-algebras are nuclear and even all ultragraph $C^*$-algebras are nuclear as each ultragraph $C^*$-algebra is Morita equivalent to a graph $C^*$-algebra, see Prop. \ref{MoritaUltra}. For hypergraph $C^*$-algebras this is not the case  as we will see in the following. In analogy to the group case, we introduce the following notion.

\begin{definition}
    We call a hypergraph \emph{amenable}, if the corresponding hypergraph $C^*$-algebra is nuclear.
\end{definition}

\subsection{Basic examples of non-amenable hypergraphs}

In the following, we view the algebra $C(S^1)$ of continuous functions on the circle as the universal $C^*$-algebra $C(S^1)\cong C^*(u,1\vert uu^*=1=u^*u)$ and $\mathbb{C}^n\cong C^*(p_i, i=1,...,n \vert p_i^*=p_i=p_i^2 ; \sum_{i=1}^np_i=1)$. The unital (full) free product of these two $C^*$-algebras is then given by the universal $C^*$-algebra generated by a unitary $u$ and $n$ projections $p_1, \ldots,p_n$ summing up to one; the units are identified.

Note that we have
\[C(S^1)*\mathbb C^n\cong C^*(\mathbb Z)*C^*(\mathbb Z/n\mathbb Z)\cong C^*(\mathbb Z * (\mathbb Z/n\mathbb Z)),\]
and thus, this $C^*$-algebra is not nuclear (\cite[Thm. 2.6.8]{brown2008calgebras}) as long as $n\geq 2$, since $\mathbb Z * (\mathbb Z/n\mathbb Z)$ is not amenable; the latter group contains the free group on two generators, see \cite[Lemma 3.1.7]{Weber_DA}, which is non-amenable, and we  use that  closed subgroups of amenable, compact groups are amenable, \cite[Thm. 1.2.7]{amenable_banach_algebras}.

\begin{proposition}
\label{productgraph}
Let $n\in\mathbb{N}$ and consider the hypergraph $\tilde{H\Gamma_n}$ with vertices $\{v_1,\dots,v_n\}$, edges $\{e_1\}$ and  $r(e_1)=\{v_1,\dots,v_n\}, \ s(e_1)=\{v_1,\dots,v_n\}$.   We have $C^*(\tilde{H\Gamma_n})\cong C(S^1)*_\mathbb{C}\mathbb{C}^n$. In particular, $C^*(\tilde{H\Gamma_n})$ is not nuclear and $\tilde{H\Gamma}_n$ is non-amenable as soon as $n\geq 2$.
\begin{figure}[h!]
        \centering
            \begin{center}
    \begin{minipage}{\linewidth}
    \centering
        \begin{tikzpicture}[ > = stealth, auto, thick]
            \node (v1) at (1.5,1.5) {$v_1$};
            \node (v2) at (3,1) {$v_2$};
            \node (v3) at (3,-1) {$\dots$};
            \node (v4) at (1.5,-1.5) {$v_{n-1}$};
            \node (v5) at (0,0) {$v_n$};
            \path[-] (v1) edge (v2); 
            \path[-] (v1) edge (v3); 
            \path[-] (v1) edge (v4); 
            \path[-] (v1) edge (v5); 
            \path[-] (v2) edge (v3); 
            \path[-] (v2) edge (v4); 
            \path[-] (v2) edge (v5); 
            \path[-] (v3) edge (v4); 
            \path[-] (v3) edge (v5); 
            \path[-] (v4) edge (v5); 
        \end{tikzpicture}
    \end{minipage}
\end{center}
        \caption[Hypergraph $\tilde{H\Gamma_n}$ generating $C(S^1)*\C^n$.]{Hypergraph $\tilde{H\Gamma_n}$ generating $C(S^1)*\C^n$. To simplify the visualization of the hypergraphs, we omit the arrowheads for edges that point in both directions.}  
    \end{figure}
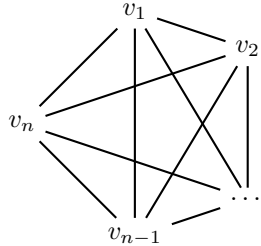
\end{proposition}

\begin{proof}
For the hypergraph $C^*$-algebra $C^*(\tilde{H\Gamma_n})$ we obtain $s_{e_1}^*s_{e_1}=\sum_{i=1}^n p_{v_i}=1$ by Relation (HR1) and Prop. \ref{unit}.  Further, by Relations (HR2a) and (HR2b) we have $s_{e_1}s_{e_1}^*\leq \sum_{i=1}^n p_{v_i}=1$ and $s_{e_1}s_{e_1}^*\geq p_{v_i}$ for all $i=1,...,n$. This implies $s_{e_1}s_{e_1}^*=\sum_{i=1}^np_{v_i}=1$. Hence, $s_{e_1}$ is a unitary and we  obtain a *-homomorphism  $\phi: C(S^1)*_\mathbb{C} \mathbb{C}^n\rightarrow C^*(\tilde{H\Gamma_n})$, sending $u$ to $s_{e_1}$ and $p_i$ to $p_{v_i}$ for all $i=1,\dots,n$. It is an isomorphism, since conversely $u, p_1, \ldots, p_n\in C(S^1)*\mathbb C^n$ satisfy the relations of $C^*(\tilde{H\Gamma_n})$.
\end{proof}

We now extend the above example to a hypergraph with $m\geq 2$ edges.

Recall the definition of the Cuntz algebra $\mathcal{O}_m=C^*(S_1,\dots,S_m \ \vert \ S_i^*S_i=1 \ \text{for all} \ i=1,\dots,m; \sum_{i=1}^m S_iS_i^*=1)$, see \cite{Cuntz}.

\begin{proposition} \label{productgraph2}
Let $n,m\in\mathbb{N}$ with $m\geq 2$. Consider the hypergraph $H\Gamma$ with 
vertices $\{v_1,...,v_n\}$, edges $\{e_1,...,e_m\}$ and $r(e_i)=\{v_1,...,v_n\}, \ s(e_i)=\{v_1,...,v_n\}$ for all $i=1,...,m$. We have $C^*(H\Gamma)\cong \mathcal{O}_m*_\mathbb{C}\mathbb{C}^n$ and hence $H\Gamma$ is non-amenable, if $n\geq 2$.
\end{proposition}

\begin{proof}
Using the relations of the associated hypergraph $C^*$-algebra $C^*(H\Gamma)$, we obtain
\begin{align*}
    &s_{e_i}^*s_{e_j}=\delta_{ij}\sum_{k=1}^n p_{v_k}=\delta_{ij} \ \text{for all} \ i,j=1,...,m\\
    &s_{e_i}s_{e_i}^*\leq \sum_{j=1}^n p_{v_j}=1 \ \text{for all} \ i=1,...,m\\
    &p_{v_i} \leq \sum_{j=1}^ms_{e_j}s_{e_j}^*\ \text{for all} \ i=1,...,n.\\ 
\end{align*}
Hence, the elements $s_{e_i}$ are isometries with $1=\sum p_{v_i}\leq\sum_j s_{e_j}s_{e_j}^*\leq 1$ and we obtain a *-homomorphism
$\phi:\mathcal{O}_m*_\mathbb{C}\mathbb{C}^n\rightarrow C^*(H\Gamma)$ that sends $S_i$ to $s_{e_i}$ for all $i=1,...,m$ and $p_i$ to $p_{v_i}$ for all $i=1,...,n$.
Conversely, the Relations (HR1)-(HR2b) are satisfied in $\mathcal{O}_m*_\mathbb{C}\mathbb{C}^n$, and we obtain a *-homomorphism that is inverse to $\phi$.
\end{proof}

\subsection{Construction of non-amenable hypergraphs}

We can now use the above non-amenable hypergraph $\tilde{H\Gamma}_n$ from Prop. \ref{productgraph} to construct further non-amenable hypergraphs. This can be achieved by extending the hypergraph $\tilde{H\Gamma}_n$ appropriately. The canonical generators of  $C^*(\tilde{H\Gamma}_n)$ are in the following denoted by $\{t_f\}\cup\{q_{v_1},\dots, q_{v_n}\}$.
The crucial idea is then to use that nuclearity transfers to quotients \cite[Cor. IV.3.1.13]{blackadar_operator}.

\begin{proposition}\label{construction_nonnuclear_1}
    Let $n\geq 3$. Let $H\Gamma$ be the hypergraph defined by $V=\{w,v_1,\dots,v_n\}$ and $E=\{e,f\}$ with 
        \begin{align*}
            s(e)&=\{w\},\qquad \qquad \;\;\; r(e)=\{v_n\},\\
            s(f)&=\{v_1,\dots, v_n\}, \qquad r(f)=\{v_1,\dots, v_n\}.
        \end{align*}
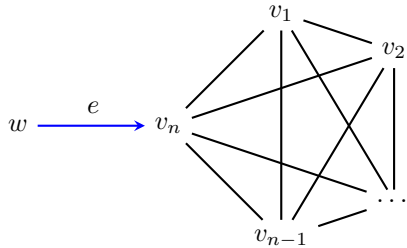
    \begin{figure}[h!]
    \centering
         \begin{center}
    \begin{minipage}{\linewidth}
    \centering
        \begin{tikzpicture}[ > = stealth, auto, thick]
            \node (v0) at (0,0) {$w$};
            \node (v1) at (3.5,1.5) {$v_1$};
            \node (v2) at (5,1) {$v_2$};
            \node (v3) at (5,-1) {$\dots$};
            \node (v4) at (3.5,-1.5) {$v_{n-1}$};
            \node (v5) at (2,0) {$v_n$};
            \path[->, blue] (v0) edge (v5); 
            \node (e0) at (1,0.25) {$e$};
            \path[-] (v1) edge (v2); 
            \path[-] (v1) edge (v3); 
            \path[-] (v1) edge (v4); 
            \path[-] (v1) edge (v5); 
            \path[-] (v2) edge (v3); 
            \path[-] (v2) edge (v4); 
            \path[-] (v2) edge (v5); 
            \path[-] (v3) edge (v4); 
            \path[-] (v3) edge (v5);  
            \path[-] (v4) edge (v5); 
        \end{tikzpicture}
    \end{minipage}
\end{center}
\end{figure}

    Then $H\Gamma$ is non-amenable.
\end{proposition}

\begin{proof}
    A Cuntz-Krieger $H\Gamma$-family in $C^*(\tilde {H\Gamma}_{n-1})$ is defined by 
        \begin{align*}
            &P_w :=0,\\
            &P_{v_n}:=0,\\
            &P_{v_i}:= q_{v_i} \qquad \text{for} \; i\leq n-1,\\
            &S_e:=0,\\
            &S_f:=t_f,
        \end{align*}
    which follows by straightforward calculations.
    This leads to a surjective *-homomorphism $\phi:C^*(H\Gamma) \rightarrow C^*(\tilde {H\Gamma}_{n-1})$ 
    By Proposition \ref{productgraph} we know that the $C^*$-algebra $C^*(\tilde {H\Gamma}_{n-1})$ is non nuclear. Thus, $C^*(H\Gamma)$ has a non-nuclear quotient and is thus also non-nuclear. 
\end{proof}

\begin{remark}
There are multiple ways to  create further non-amenable hypergraphs with the above technique. We could add multiple vertices to the source/range of $e$, attach more edges to the hypergraph, and so on. The main idea is always to set the partial isometry corresponding to the new edge equal to zero and use the hypergraph relations to determine which projections must be zero. Then we consider the hypergraph $\tilde{H\Gamma}_m$ with edge $f$, were we delete all vertices, whose projection is zero, from the range and source. The resulting Cuntz-Krieger family leads to a surjective *-homomorphism onto a non-nuclear $C^*$-algebra $C^*(\tilde{H\Gamma}_m)$. We add further such examples in Appendix \ref{Appendix_Non_Amenable_hypergraphs}.
\end{remark}
Nevertheless, the above illustration is somewhat misleading. The examples originate from manipulations of the non-amenable hypergraph. Thus, at first glance, the non-amenable subhypergraph $\tilde{H\Gamma}_n$ seems to be crucial. But in fact the quotient given by $C^*(\tilde{H\Gamma}_m)$ for some $m <n$ is crucial, as highlighted in the following figure:

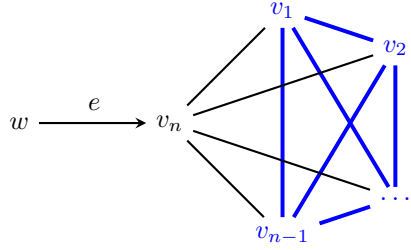
\begin{figure}[h!]
    \centering
        \begin{center}
    \begin{minipage}{\linewidth}
    \centering
        \begin{tikzpicture}[ > = stealth, auto, thick]
            \node (v0) at (0,0) {$w$};
            \node (v1) at (3.5,1.5) {$\textcolor{blue}{v_1}$};
            \node (v2) at (5,1) {$\textcolor{blue}{v_2}$};
            \node (v3) at (5,-1) {$\textcolor{blue}{\dots}$};
            \node (v4) at (3.5,-1.5) {$\textcolor{blue}{v_{n-1}}$};
            \node (v5) at (2,0) {$v_n$};
            \path[->] (v0) edge (v5); 
            \node (e0) at (1,0.25) {$e$};
            \path[-, blue, line width=1.5pt] (v1) edge (v2); 
            \path[-, blue, line width=1.5pt] (v1) edge (v3); 
            \path[-, blue, line width=1.5pt] (v1) edge (v4); 
            \path[-] (v1) edge (v5); 
            \path[-, blue, line width=1.5pt] (v2) edge (v3); 
            \path[-, blue, line width=1.5pt] (v2) edge (v4); 
            \path[-] (v2) edge (v5); 
            \path[-, blue, line width=1.5pt] (v3) edge (v4); 
            \path[-] (v3) edge (v5); ; 
            \path[-] (v4) edge (v5); 
        \end{tikzpicture}
    \end{minipage}
\end{center}
\end{figure}

To get a better graphical understanding we express the above technique by concrete requirements on the hypergraph. The idea is to extract a non-amenable part of the hypergraph with slight modifications on the edges by deleting vertices from its source and range. This gives an easy way to check non-amenability for a given hypergraph without using the corresponding $C^*$-algebra.

\begin{proposition}
    Let $H\Gamma=(V, E, r,s)$ be a finite hypergraph. Assume there exist $\tilde V\subseteq V$ and $\tilde E\subseteq E$ such that $\tilde V\cap r(e)\neq \emptyset$ and $\tilde V \cap s(e)\neq \emptyset$ holds if and only if $e\in \tilde E$. Let  $\tilde{H\Gamma}=(\tilde V, \tilde E, \tilde r, \tilde s)$ be the hypergraph with
        \begin{align*}
            &\tilde r(e):=r(e)\cap \tilde V,\\
            &\tilde s(e):=s(e)\cap \tilde V.
        \end{align*}
    Then $C^*(\tilde{H\Gamma})$ is a quotient of $C^*(H\Gamma)$. In particular, if $\tilde{H\Gamma}$
    is non-amenable, then $H\Gamma$ is non-amenable, since nuclearity passes to quotients.
\end{proposition}

\begin{proof}
     We can define a Cuntz-Krieger $H\Gamma$-family in $C^*(\tilde{H\Gamma})$ by
        \begin{align*}
            P_v&:=\begin{cases}  q_v &\text{for} \; v \in \tilde V\\0 & \text{for} \; v \in V\setminus \tilde V,\end{cases}\\
            S_e&:=\begin{cases} t_e &\text{for}\; e \in \tilde E\\ 0 & \text{for} e \in E\setminus \tilde E.\end{cases}
        \end{align*}
    Since all edges whose sources and sinks intersect with $\tilde V$ lie in $\tilde E$, the hypergraph relations of the Cuntz-Krieger family  follow directly from the hypergraph relations in $C^*(\tilde{H\Gamma})$. Indeed, as the projections and partial isometries corresponding to vertices and edges not in $\tilde{H\Gamma}$ are 0, they can be added without changing any relations. By the universal property we obtain a *-homomorphism  from $C^*(H\Gamma)$ onto $C^*(\tilde{H\Gamma})$ which is surjective as all generators of $C^*(\tilde{H\Gamma})$ lie in its range.
\end{proof}

The previous constructions have the objective to check if a given hypergraphs is non-amenable by deleting and manipulating edges and vertices. The emerging question is now, how to attach a non-amenable hypergraph to an arbitrary hypergraph to receive a non-amenable hypergraph. The technique below defines some kind of product between two hypergraph $C^*$-algebras.

\begin{proposition}
    Let $H\Gamma=(V_\Gamma,E_\Gamma,r_\Gamma,s_\Gamma)$ and $H\Delta=(V_\Delta,E_\Delta,r_\Delta,s_\Delta)$ be finite hypergraphs. For fixed $f \in E_\Gamma$ and $w \in V_\Delta$ let the hypergraph $H\Theta$ be given by
        \begin{align*}
            &V_\Theta=V_\Gamma\cup V_\Delta,\\
            &E_\Theta=E_\Gamma \cup E_\Delta,\\
            &r_\Theta(e)=\begin{cases} r_\Gamma(e) & \text{for} \;e \in E_\Gamma\\ r_\Delta(e) & \text{for}\; e \in E_\Delta,\end{cases}\\
            &s_\Theta(e):=\begin{cases} s_\Gamma(e) & \text{for} \;e \in E_\Gamma\setminus \{f\}\\ s_\Gamma(f)\cup \{w\} & \text{for} \; e=f\\ r_\Delta(e) & \text{for}\; e \in E_\Delta.\end{cases}\\
        \end{align*}
    Then $C^*(H\Gamma)$ is a quotient of $C^*(H\Theta)$. In particular,
  if $H\Gamma$ is non-amenable, then $H\Theta$ is non-amenable.
\end{proposition}

\begin{proof}
    We define a Cuntz-Krieger $H\Theta$-family in $C^*(H\Gamma)$ by letting all projections and partial isometries corresponding to vertices and edges in $H\Delta$ be zero and by identifying the elements corresponding to vertices and edges in $H\Gamma$ with the generators of $C^*(H\Gamma)$. The clue is, that by letting all elements corresponding to $C^*(H\Delta)$ be zero, we "delete" the new vertex in the source and obtain $H\Gamma$. We denote the elements in the constructed Cuntz-Krieger $H\Theta$-family by $T_e$ and $Q_v$. This is indeed a Cuntz-Krieger $H\Theta$-family. The crucial part is the linking edge $f$ and the vertex $w$.  If we consider the hypergraph relations for these we get: 
    \begin{align*}
        T_f^*T_f&=s_f^*s_f=p_{r_\Gamma(f)}=p_{r_\Theta(f)}=Q_{r_\Theta(f)},\\
        T_fT_f^*&=s_fs_f^*\leq p_{s_\Gamma(f)}=p_{s_\Gamma(f)}+0=Q_{s_\Gamma(f)}+Q_w=Q_{s_\Theta(f)},\\
        Q_w&=0\leq \sum_{e \in E_\Theta, w \in s_\Theta(e)}T_eT_e^*.
    \end{align*}
    The *-homomorphism  given by the universal property is clearly surjective, as all generators of $C^*(H\Gamma)$ are in the range.
\end{proof}

\begin{remark}
In the previous proposition, we added the vertex $w$ to the source of the edge $f$. Similarly we could have also added the vertex $w$ to the range of  $f$. In either cases the quotient deletes the further vertex in the source/range. Furthermore, we must not restrict ourselves to a single connection. Using the same idea of the proof we could extend to multiple linking edges and multiple new vertices in their sources/ranges. 
\end{remark}

\section{\textbf{Gauge-invariant uniqueness for hypergraph $C^*$-algebras}}
\label{sec:gauge}

We now take a look at  the Gauge-Invariant Uniqueness Theorem, which -- for graph $C^*$-algebras -- yields faithful representations.  For this class, it is one of the most important theorems  and it holds for all graph $C^*$-algebras \cite{ideal_structure_infinite_graphs}[Thm. 2.1] and all ultragraph $C^*$-algebras \cite{Tomforde}[Prop. 5.5]. However, this is not the case for hypergraph $C^*$-algebras, as we will see in this section. Nevertheless, under specific assumptions on the hypergraph, we can extend it, borrowing the Gauge-Invariant Uniqueness Theorem from graph $C^*$-algebras.
Let us first introduce gauge actions.

\begin{proposition}
    Let $H\Gamma=(V, E, r,s)$ be a finite hypergraph with universal Cuntz-Krieger $H\Gamma$-family $\{s_e,p_v\}$. Then there exists a continuous action of $\T$ on $C^*(H\Gamma)$ such that 
    \begin{align*} \label{gaugeaction}
        &\gamma_z(s_e)=zs_e \quad \forall e \in E, \hspace{30pt} \gamma_z(p_v)=p_v \quad \forall v\in V.
    \end{align*}
    The action is called \emph{gauge action}\index{Gauge! action}.
\end{proposition}

\begin{proof}
    The proof is analogous to the one for graphs \cite[Prop. 2.1]{Raeburn}: One can immediately see, that the elements $zs_e$ and $p_v$ satisfy the Cuntz-Krieger relations for hypergraphs, which by the universal property yields the *-isomorphisms $\gamma_z:C^*(H\Gamma)\to C^*(H\Gamma)$ with $\gamma_w\circ \gamma_z=\gamma_{wz}$. One is left to check continuity as in \cite[Prop. 2.1]{Raeburn}.
\end{proof}

Now, for graph $C^*$-algebras, the Gauge-Invariant Uniqueness Theorem can be formulated as follows.

\begin{theorem}[Gauge-Invariant Uniqueness Theorem for graph $C^*$-algebras]\label{gaugeforgraphs}
Let $\Gamma=(V,E,r,s)$ be a finite graph and let $\pi:C^*(\Gamma)\to B$ be a *-homomorphism, where $B$ is some $C^*$-algebra. If there exists an action $\beta:\T\to Aut(B)$ such that $\beta_z\circ\pi=\pi\circ\gamma_z$, for all $z\in \T$, and if $\pi(p_v)\neq 0$ for all $v\in V$, then $\pi$ is injective.  
\end{theorem}

This theorem does not apply to hypergraphs in general, as we can see in the next example. 

\begin{beispiel}\label{failed_gauge_uniqueness}
Consider  the hypergraph $\tilde{H\Gamma}_n$ from Prop. \ref{productgraph}. Its corresponding hypergraph $C^*$-algebra is given by $C^*(\tilde{H\Gamma}_n)\cong C(S^1)*\C^n$. We thus obtain a surjective but non-injective *-homomorphism 
    \begin{align*}
        \pi: C^*(\tilde{H\Gamma}_n)\cong C(S^1)*\C^n \rightarrow C(S^1)\otimes \C^n.
    \end{align*}
    The tensor product $C(S^1)\otimes \C^n$ has a gauge action $\beta$ with the property, that $\beta\circ \pi=\pi \circ \gamma$, where $\gamma$ is the gauge action on $C^*(\tilde{H\Gamma}_n)$. Thus, the Gauge-Inivariant Uniqueness Theorem does not hold for this hypergraph $C^*$-algebra.
\end{beispiel}

We will now show that under specific conditions, a hypergraph $C^*$-algebra is isomorphic to the $C^*$-algebra of its dual graph. For graphs, this is true in general, see \cite[Ex. 2.7]{Raeburn}: The $C^*$-algebra of a graph is isomorphic to the $C^*$-algebra of its  dual graph, if the graph has no sinks.

\begin{definition}
    Let $H\Gamma=(V, E, r, s)$ be a finite hypergraph. The \emph{dual graph}\index{Dual Graph} $\tilde \Gamma$ is defined as 
    \begin{align*}
        &\tilde{V}:=\{e \; |\; e \in E\},\\
        &\tilde{E}:=\{(e,f)\; |\; e,f\in E, \; s(f)\cap r(e)\neq \emptyset\},\\
        &\tilde s((e,f)):=e,\\
        &\tilde r((e,f)):=f.
    \end{align*}
\end{definition}

The dual graph is actually a graph, not just a hypergraph. 

\begin{lemma}\label{rangeprojectionquasiperfectpaths}
    Let $H\Gamma=(V, E, r,s)$ be a finite hypergraph with only quasi perfect paths and no sinks. Then it holds for all $e \in E$
        \begin{align*}
            p_{r(e)}=\sum_{f\in E,\; s(f)\subseteq r(e)}s_fs_f^*. 
        \end{align*}
\end{lemma}

\begin{proof}
    By definition of quasi perfect paths,  $s(f)\cap r(e)\neq \emptyset$ implies $s(f) \subseteq r(e)$ for all $e,f \in E$. The statement follows when combining (HR2a) and (HR2b). 
\end{proof}

\begin{proposition} \label{dualgraphisom}
    Let $H\Gamma=(V, E, r ,s)$ be a finite hypergraph with only quasi perfect paths and no sinks, $\tilde \Gamma$ be its dual graph. Then $C^*(\tilde \Gamma)$ is  isomorphic to the $C^*$-subalgebra of $C^*(H\Gamma)$ generated by $\{s_e\;|\; e\in E\}$.
\end{proposition}

\begin{proof}
    We define a Cuntz-Krieger $\tilde \Gamma$-family in $C^*(H\Gamma)$ by 
        \begin{align*}
            Q_e&:=s_es_e^*\hspace{35pt} \forall e \in \tilde{V},\\
            T_{(e,f)}&:=s_eQ_f \hspace{30pt} \forall (e,f) \in \tilde{E}.
        \end{align*}
    To check that this is a Cuntz-Krieger $\tilde \Gamma$-family, a short calculation using the hypergraph relations of $C^*(H\Gamma)$ shows, that the elements $Q_e$ are mutually orthogonal projections and the elements $T_{(e,f)}$ are partial isometries. 
    It remains to check the Cuntz-Krieger relations for graphs, since the dual graph is a graph.
    Since the hypergraph has only quasi perfect paths, all paths $ef$ fulfill $s(f)\subseteq r(e)$ which implies by Lemma \ref{source_range_projections} and the definition of $Q_f$, that $Q_fp_{r(e)}=Q_f$. With this we get the first Cuntz-Krieger relation:
        \begin{align*}
            T_{(e,f)}^*T_{(g,h)}&=(s_eQ_f)^*(s_gQ_h)\\
            &=Q_fs_e^*s_gQ_h\\
            &=\delta_{e,g} Q_fp_{r(e)}Q_h\\
            &=\delta_{e,g} Q_fQ_h\\
            &=\delta_{e,g} \delta_{f,h}Q_f\\
            &=\delta_{(e,f),(g,h)} Q_{r((e,f))}.
    \intertext{For the second Cuntz-Krieger relation we need that for quasi perfect paths with no sinks, Lemma \ref{rangeprojectionquasiperfectpaths} applies, and we get:}
            Q_{e}  &=s_es_e^*\\
            &=s_ep_{r(e)}s_e^*\\
            &=s_e \left( \sum_{f\in E, s(f)\subseteq r(e)}s_fs_f^*\right)s_e^*\\
            &=s_e \left( \sum_{f\in E, s(f)\subseteq r(e)}Q_f\right)s_e^*\\
            &=\sum_{f\in E, s(f)\subseteq r(e)}s_eQ_fs_e^*\\
            &=\sum_{f\in E, s(f)\subseteq r(e)}T_{(e,f)}T_{(e,f)}^*\\
            &=\sum_{x\in \tilde{E}, s(x)=e}T_{x}T_{x}^*.
        \end{align*}
    Thus, by the universal property, we get a canonical *-homomorphism $\pi:C^*(\tilde{\Gamma}) \rightarrow C^*(H\Gamma)$ defined by $q_e \mapsto Q_e$ and $t_{(e,f)}  \mapsto T_{(e,f)}$.
    Since  the dual graph is really a graph, we  can use the Gauge-Invariant Uniqueness Theorem for graph $C^*$-algebras, Thm. \ref{gaugeforgraphs}.
    Let $\gamma$ and $\tilde{\gamma}$ be the gauge actions on $C^*(H\Gamma)$ and  $C^*(\tilde{\Gamma})$ respectively. Then a short calculation shows, that $\pi \circ \tilde{\gamma}_z=\gamma_z \circ \pi$ for all $z \in \T$. Thus the requirements for the Gauge-Invariant Uniqueness Theorem are given (note that all projections are non-zero) and the *-homomorphism $\pi$ is injective. By definition of the Cuntz-Krieger family $\{T_{(e,f)}, Q_f\}$ we know that $Im(\pi)\subseteq C^*(s_e\;|\; e\in E)$. 
    Using again Lemma \ref{rangeprojectionquasiperfectpaths} we get
        \begin{align*}
            s_e=s_ep_{r(e)}=s_e\sum_{f\in E, s(f)\subseteq r(e)}s_fs_f^*=\sum_{f\in E, s(f)\subseteq r(e)}s_eQ_f=\sum_{f\in E, s(f)\subseteq r(e)}T_{(e,f)}.
        \end{align*}
    Hence the Cuntz-Krieger family $\{T_{(e,f)}, Q_f\}$ generates $\{s_e\;|\; e\in E\}$. Thus, $Im(\pi)=C^*(s_e\;|\; e\in E)$ and $\pi$ is an isomorphism between $C^*(\tilde{\Gamma})$ and $C^*(s_e\;|\; e\in E)$.
\end{proof}

\begin{korollar}\label{HyperIsomorphGraph}
    Let $H\Gamma=(V, E, r, s)$ be a finite hypergraph with only quasi perfect paths, no sinks and $C^*(H\Gamma)$ be generated by $\{s_e\;|\; e\in E\}$, $\tilde \Gamma$ be its dual graph. Then $C^*(\tilde \Gamma)\cong C^*(H\Gamma)$.
\end{korollar}

For specific hypergraphs we thus get an isomorphism between the hypergraph $C^*$-algebra and the graph $C^*$-algebra of its dual graph. This cannot be true in general -- this would not be in line with the non-nuclear examples of hypergraph $C^*$-algebras. 
Also, there cannot be a non-amenable hypergraph with only quasi perfect paths and no sinks such that $C^*(H\Gamma)$ is generated by $\{s_e\;|\; e\in E\}$.

As an immediate consequence of the above corollary, we may formulate a Gauge-Invariant Uniqueness Theorem for hypergraph $C^*$-algebras under the above specific assumptions on the hypergraph.

\begin{theorem}[Gauge-Invariant Uniqueness Theorem for hypergraph $C^*$-algebras]\label{Gauge_Uniqueness_hypergraphs} 
    Let $H\Gamma=(V, E, r, s)$ be a finite hypergraph with only quasi perfect paths, no sinks and $C^*(H\Gamma)$ be generated by $\{s_e\;|\; e\in E\}$. Let $\{P_v,S_e\}$ be a Cuntz-Krieger $H\Gamma$-family in a $C^*$-algebra $B$ with each $P_v\neq 0$. If there is a continuous action $\beta: \T \rightarrow Aut(B)$ such that the gauge action $\gamma$ commutes with the canonical *-homomorphism $\pi:C^*(H\Gamma) \rightarrow B$, i.e. $\pi\circ \gamma_z=\beta_z \circ \pi$ for all $z \in \T$. Then $\pi$ is injective.
\end{theorem}

\begin{proof}
  We use Cor. \ref{HyperIsomorphGraph} and the Gauge-Invariant Uniqueness Theorem for graphs.  
\end{proof}

The above version of a Gauge-Invariant Uniqueness Theorem for hypergraph $C^*$-algebras is not very deep, as it directly relies on the Gauge-Invariant Uniqueness Theorem for graph $C^*$-algebras. However, it can be taken as a starting point for proving more general Gauge-Invariant Uniqueness Theorems for hypergraph $C^*$-algebras -- or rather for specifying the class of hypergraphs for which it holds. It seems that for the class of hypergraph $C^*$-algebras, it is more appropriate to speak of a Gauge-Invariant Uniqueness Property and to determine the class of hypergraphs possessing it.

Finally, let us derive from Prop. \ref{dualgraphisom} a proof of the Gauge-Inivariant Uniqueness Theorem for ultragraphs. This fact has been known before, \cite{Tomforde}[Prop. 5.5].

\begin{korollar}[Gauge-Invariant Uniqueness Theorem for ultragraph $C^*$-algebras]
    Let $\mathcal{G}$ be an ultragraph without sinks. Then the Gauge-Invariant Uniqueness Theorem is valid.
\end{korollar}

\begin{proof}
    We check that ultragraphs fulfill the assertions of Theorem \ref{Gauge_Uniqueness_hypergraphs}. Since the source of each edge in an ultragraph is given by one vertex, combining the second and third hypergraph relations yields $p_v=\sum_{e \in E, v =s(e)}s_es_e^*$. Since all vertices emit at least one edge, this equality is valid for all vertices and hence, the $C^*$-algebra is generated by the partial isometries. Furthermore, each path is automatically quasi perfect, as $r(e)\cap s(f)\neq 0$ implies that $r(e)\cap s(f)=s(f)$, as $s(f)$ consists of just one vertex. 
\end{proof}

\begin{remark} \label{ultragraphs_restrictions_gauge_uniqueness}
    As seen in Theorem \ref{Gauge_Uniqueness_hypergraphs}, the Gauge-Uniqueness Theorem holds for finite hypergraphs with only quasi perfect paths, no sinks and whose C*-algebra is generated by the partial isometries $s_e$. Do hypergraphs with these  properties exist that are not ultragraphs -- or are ultragraphs without sinks exactly the hypergraphs with quasi perfect paths whose $C^*$-algebra is generated by its generating partial isometries? This could be interesting for further research. 
\end{remark}

\section{\textbf{Moves on hypergraphs}}
\label{sec:moves}

In this section we discuss basic moves to manipulate hypergraphs. These moves play an important role in the classification of graph $C^*$-algebras up to stable isomorphism, see  \cite{complete_classification}. We introduce four of these moves, adapt them to the hypergraph setting and investigate the corresponding $C^*$-algebras. The construction of the moves is motivated from graph $C^*$-algebras, see for instance \cite[Def. 2.14--2.17]{geometric_finite_graphs}. For readers familiar with the theory of symbolic dynamics note that the moves we consider are closely related to flow equivalence of shifts spaces \cite{lind_marcus_2021}. 

The main result of this section is, that as soon as these moves are performed at vertices which locally behave like graphs (or ultragraphs), we observe a behaviour similar to the moves for graphs.

\begin{definition}
    A hypergraph is called \emph{locally ultra at vertex $w$}, if  for all $e\in E$, the assertion $w\in s(e)$ implies $s(e)=\{w\}$.
\end{definition}

Recall from Section 3.7 that $\lvert r(e)\rvert\geq 2$ is a feature we can circumvent, by decomposition of ranges. So, it is really the property $\lvert s(e)\rvert\geq 2$, which separates hypergraphs from ultragraphs ($\lvert s(e)\rvert=1$) and graphs ($\lvert s(e)\rvert=\lvert s(e)\rvert=1$). In other words, if $w\in s(e)$ implies $\lvert s(e)\rvert=1$, then the graph is ``like an ultragraph/graph'' at vertex $w$, indeed.

The following theorem summarizes the main contents of this section; it will be proven step by step in the sequel.

\begin{theorem}\label{thmmovessummary}
    Let $H\Gamma$ be a hypergraph, and let $H\Gamma$ be locally ultra at some  vertex $w$ (possibly satisfying some further mild assumptions on $w$). The moves S, R, I and O produce hypergraphs $H\Gamma'$ and *-homomorphisms $\pi:C^*(H\Gamma')\to C^*(H\Gamma)$ respectively, such that $Im(\pi)$ is a full corner in $C^*(H\Gamma)$. In the case of Move O, $\pi$ is even a *-isomorphism.
\end{theorem}

The definition of the moves, the precise statements of the theorem for the respective moves, and the proof of this theorem will be given in the subsequent subsections. Note that unlike in the graph case, the *-homomorphisms $\pi$ for moves S, R and I are not necessarily injective. Hence, we only have a weaker form of Morita equivalence in the hypergraph case. In the graph and ultragraph cases, the Gauge-Invariant Uniqueness Theorem can be applied in order to obtain injectivity of $\pi$, so again, the Gauge-Invariant Uniqueness Property for hypergraphs plays a role here, when speaking about moves.

\subsection{Definition of the moves}

\begin{definition}[\textbf{Move S}]
    Let $H\Gamma=(V, E, r,s)$ be a finite hypergraph. Let $w \in V$ be a source. The hypergraph $H\Gamma_S$ obtained by application of \emph{move S}\index{Move!S - Removing a source} is defined as 
    \begin{align*}
        V_S:=V\setminus \{w\}, \quad
        E_S:=E\setminus \{e\;|\; w \in s(e)\},\quad
        s_S:=s_{|E_S},\quad
        r_S:=r_{|E_S}.
    \end{align*}
We call $H\Gamma_S$ the hypergraph obtained by removing the source $w$ from the hypergraph $H\Gamma$. 
\end{definition}

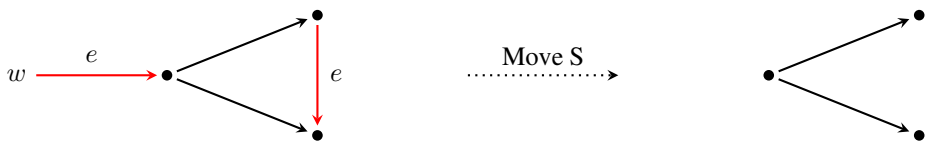
\begin{figure}[h!]
    \centering
        \begin{center}
        \begin{minipage}{\linewidth}
        \centering
            \begin{tikzpicture}[ > = stealth, auto, thick]
                \node (v0) at (0,0) {$w$};
                \node (v1) at (2,0) {};
                \fill (v1) circle[radius=2pt];
                \node (v2) at (4,0.8) {};
                \fill (v2) circle[radius=2pt];
                \node (v3) at (4,-0.8) {};
                \fill (v3) circle[radius=2pt];
                \path[->, red] (v0) edge (v1); 
                \node (e0) at (1,0.25) {$e$};
                \path[->] (v1) edge (v2); 
                \node (e1) at (2.9,0.7) {};
                \path[->] (v1) edge (v3); 
                \node (e2) at (2.9,-0.7) {};
                \path[->, red] (v2) edge (v3); 
                \node (e2) at (4.25,0) {$e$};
                
                \path[->, dotted] (6,0) edge (8,0); 
                \node (e2) at (7,0.25) {Move S};
                
                \node (v1) at (10,0) {};
                \fill (v1) circle[radius=2pt];
                \node (v2) at (12,0.8) {};
                \fill (v2) circle[radius=2pt];
                \node (v3) at (12,-0.8) {};
                \fill (v3) circle[radius=2pt];
                \path[->] (v1) edge (v2); 
                \node (e1) at (10.9,0.7) {};
                \path[->] (v1) edge (v3); 
                \node (e2) at (10.9,-0.7) {};
            \end{tikzpicture}
        \end{minipage}
    \end{center}
    \vspace{-10pt}
    \caption{Illustration of the application of move S for hypergraphs.}
\end{figure}

\begin{definition}[\textbf{Move R}]
    Let $H\Gamma=(V, E, r,s)$ be a finite hypergraph. Let $w \in V$ be a vertex that emits exactly one edge $f$ and only one vertex $x\neq w$ emits to $w$. The hypergraph $H\Gamma_R$ obtained by application of \emph{move R}\index{Move! R - Reduction at a non-sink} is defined as 
    \begin{align*}
        &V_R:=V\setminus \{w\}, \\
        &E_R:=E\setminus \left(r^{-1}(\{w\}) \cup \{f\}\right) \cup \left\{e_f\;|\; e \in E, r(e)=\{w\}\right\},\\
        &s_R(e)=s(e), \qquad s_R(e_f)=s(e),\\
        &r_R(e)=r(e), \qquad r_R(e_f)=(r(e)\setminus\{w\})\cup r(f).
    \end{align*}
\end{definition}

\begin{figure}[h!]
    \centering
       \begin{center}
        \begin{minipage}{\linewidth}
        \centering
            \begin{tikzpicture}[ > = stealth, auto, thick]
                \node (v0) at (0,0) {$x$};
                \node (v1) at (2,-0.5) {$w$};
                \node (v2) at (4,0) {};
                \fill (v2) circle[radius=2pt];
                \node (v3) at (4,-1) {};
                \fill (v3) circle[radius=2pt];
                \node (v4) at (4,1) {};
                \fill (v4) circle[radius=2pt];
                \path[->] (v0) edge (v1); 
                \path[->] (v0) edge (v4); 
                \path[->, blue, line width=1.5pt] (v1) edge (v2); 
                \node (e3) at (2.9,0) {$f$};
                \path[->, blue, line width=1.5pt] (v1) edge (v3); 
                \node (e4) at (2.9,-1) {$f$};
                
                \path[->, dotted] (6,0) edge (8,0); 
                \node (e2) at (7,0.25) {Move R};
                
                \node (v4) at (10,0) {$x$};
                \node (v5) at (12,0) {};
                \fill (v5) circle[radius=2pt];
                \node (v6) at (12,-1) {};
                \fill (v6) circle[radius=2pt];
                \node (v7) at (12,1) {};
                \fill (v7) circle[radius=2pt];
                \path[->] (v4) edge (v7);
                \path[->] (v4) edge (v5);
                \path[->] (v4) edge (v6); 
            \end{tikzpicture}
        \end{minipage}
    \end{center}
    \caption[Illustration of the application of move R for hypergraphs.]{Illustration of the application of move R for hypergraphs. Each color/thickness represents one edge.}
\end{figure}
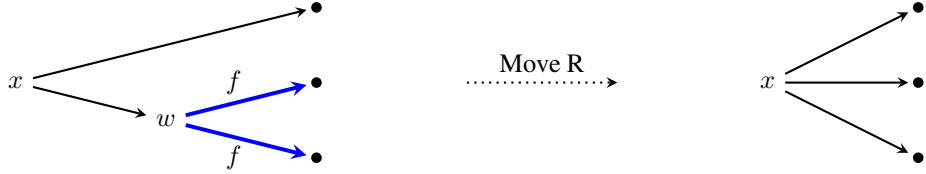

\begin{definition}[\textbf{Move O}]
    Let $H\Gamma=(V, E, r,s)$ be a finite hypergraph and $w$ be a vertex that is not a sink. We partition the set of outgoing edges in finitely many nonempty sets: 
    \begin{equation*}
        \{e \in E\; |\; w \in s(e)\}=\mathcal{E}_1 \cup \dots \cup \mathcal{E}_n.
    \end{equation*}
    The hypergraph $H\Gamma_O$ obtained by performing \emph{move O}\index{Move! O - Outsplitting} on $H\Gamma$ is defined by 
    \begin{align*}
        V_O&:=V\setminus \{w\} \cup \{w^1,\dots,w^n\},\\
        E_O&:=\{e^1 \;|\; e \in E, w \notin r(e)\} \cup \{e^1, \dots , e^n\; |\; e \in E, \; w \in r(e)\},\\
        r_O(e^i)&:=\begin{cases} r(e) &\quad \text{if} \; i=1\; \text{and}\;w \notin r(e)\\
        \left(r(e)\setminus \{w\}\right) \cup \{w^1\} &\quad \text{if} \; i=1\; \text{and}\;w \in r(e) \\ 
        w^i &\quad \text{if} \; i>1\; \text{and}\;w \in r(e),\end{cases} \\
        s_O(e^i)&:=\begin{cases} s(e) & \quad \text{if} \; w \notin s(e) \\ \left(s(e)\setminus \{w\}\right) \cup \{w^j\} &\quad \text{if} \; w \in s(e) \; \text{and}\; e \in \mathcal{E}_j. \end{cases}
    \end{align*}
    We call $H\Gamma_O$ the hypergraph obtained by \emph{outsplitting}\index{Outsplitting} 
    $H\Gamma$ at $w$.
\end{definition}

\begin{figure}[h!]
    \centering
        \begin{center}
        \begin{minipage}{\linewidth}
        \centering
            \begin{tikzpicture}[ > = stealth, auto, thick, scale=0.8]
                \node (v1) at (0,0) {};
                \fill (v1) circle[radius=2pt];
                \node (v2) at (4,0) {$w$};
                \node (v3) at (8,2) {};
                \fill (v3) circle[radius=2pt];
                \node (v4) at (8,0) {};
                \fill (v4) circle[radius=2pt];
                \node (v5) at (8,-2) {};
                \fill (v5) circle[radius=2pt];
                \path[->, bend left = 20, blue, line width=1.7pt] (v2) edge (v5);
                \path[->, bend right=20, green, line width=1.1pt] (v2) edge (v5);
                \path[->, blue, line width=1.7pt] (v2) edge (v4); 
                \path[->, orange, line width=0.4pt] (v3) edge (v4); 
                \path[->, orange, line width=0.4pt] (v2) edge (v3); 
                \path[->, in=0, out=90, loop, orange, line width=0.4pt] (v3) edge (v3);
                \path[->] (v1) edge (v2); 
                \node at (2.2,0.2) {$e$};
                \path[->] (v1) edge (v3);
                \node at (4,1.3) {$e$};

                \path[->, dashed] (9,0) edge (11,0);
                \node at (10,0.25) {Outsplitting};
                
                \node (v6) at (12,0) {};
                \fill (v6) circle[radius=2pt];
                \node (v7) at (16,0) {$w_2$};
                \node (v8) at (20,2) {};
                \fill (v8) circle[radius=2pt];
                \node (v9) at (20,0) {};
                \fill (v9) circle[radius=2pt];
                \node (v10) at (20,-2) {};
                \fill (v10) circle[radius=2pt];
                \node (v11) at (16,2) {$w_1$};
                \node (v12) at (16,-2) {$w_3$};
                \path[->, blue, line width=1.7pt] (v7) edge (v9);
                \path[->, green, line width=1.1pt] (v12) edge (v10);
                \path[->, blue, line width=1.7pt] (v7) edge (v10); 
                \path[->, orange, line width=0.4pt] (v11) edge (v8); 
                \path[->, orange, line width=0.4pt] (v8) edge (v9); 
                \path[->, in=0, out=90, loop, orange, line width=0.4pt] (v8) edge (v8);
                \path[->] (v6) edge (v11); 
                \node at (14.2,1.4) {$e^1$};
                \path[->] (v6) edge (v7);
                \node at (14.2,0.3) {$e^2$};
                \path[->] (v6) edge (v12); 
                \node at (14.2,-0.7) {$e^3$};
                \path[->, bend left=40] (v6) edge (v8); 
                \node at (16,3) {$e^1$};
            \end{tikzpicture}
        \end{minipage}
    \end{center}
    \caption[Illustration of the application of move O for hypergraphs.]{Illustration of the application of Move O for hypergraphs. Each color/thickness marks one edge and we partition the set of outgoing edges of $w$ into one-point sets.}
\end{figure}
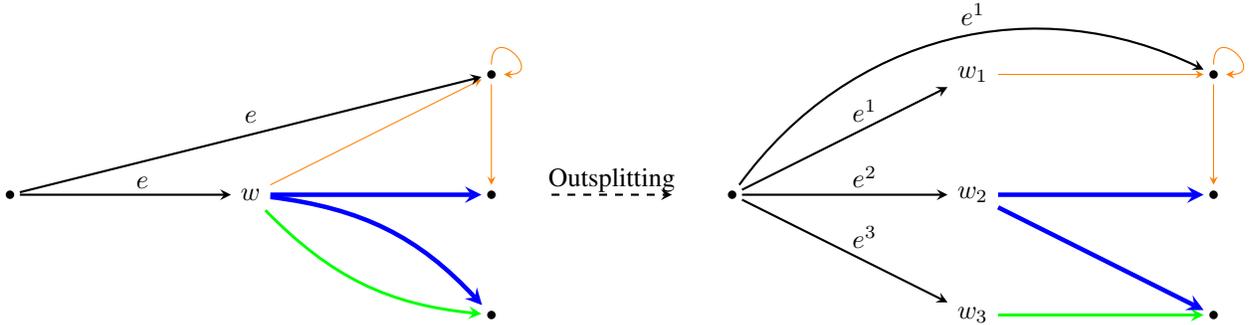

\begin{definition}[\textbf{Move I}]
    Let $H\Gamma=(V, E, r,s)$ be a finite hypergraph and $w$ be a vertex that is not a source. We partition the set of incoming edges in finitely many nonempty sets: 
    \begin{equation*}
        \{e \in E\; |\; w \in r(e)\}=\mathcal{E}_1 \cup \dots \cup \mathcal{E}_n.
    \end{equation*}
    The hypergraph $H\Gamma_I$ obtained by performing \emph{move I}\index{Move! I - Insplitting} on $H\Gamma$ is defined by 
    \begin{align*}
        V_I&:=V\setminus \{w\} \cup \{w^1,\dots,w^n\},\\
        E_I&:=\{e^1 \;|\; e \in E, w \notin s(e)\} \cup \{e^1, \dots , e^n\; |\; e \in E, \; w \in s(e)\},\\
        r_I(e^i)&:=\begin{cases} r(e) &\quad \text{if} \; i=1\; \text{and}\;w \notin r(e)\\
        \left(r(e)\setminus \{w\}\right) \cup \{w^j\} &\quad \text{if} \; e^i \in \mathcal{E}_j,\end{cases} \\
        s_I(e^i)&:=\begin{cases} s(e) & \quad \text{if} \; w \notin s(e) \\ \left(s(e)\setminus \{w\}\right) \cup \{w^1\} &\quad \text{if} \;i=1 \; \text{and}\; w \in s(e) \\
        \{w^i\} & \quad i=2,\dots, n.\end{cases} 
    \end{align*}
    We call $H\Gamma_I$ the hypergraph obtained by \emph{insplitting}\index{Insplitting} 
    $H\Gamma$ at $w$.
\end{definition}

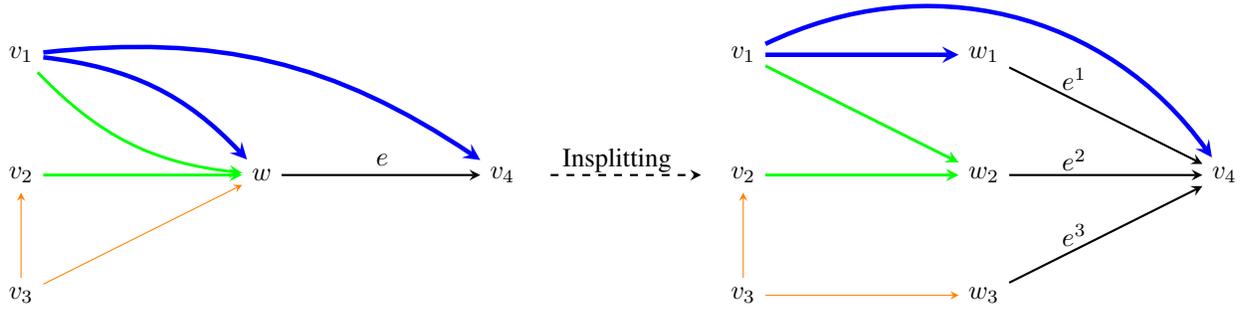
\begin{figure}[h!]
    \centering
         \begin{center}
        \begin{minipage}{\linewidth}
        \centering
            \begin{tikzpicture}[ > = stealth, auto, thick, scale=0.8]
                \node (v1) at (0,2) {$v_1$};
                \node (v2) at (0,0) {$v_2$};
                \node (v3) at (0,-2) {$v_3$};
                \node (v4) at (4,0) {$w$};
                \node (v5) at (8,0) {$v_4$};
                \node(v6) at (6,0.25) {$e$};
                \path[->, bend left = 20, blue, line width=1.7pt] (v1) edge (v4); 
                \path[->, bend right=20, green, line width=1.1pt] (v1) edge (v4); 
                \path[->, bend left = 20, blue, line width=1.7pt] (v1) edge (v5); 
                \path[->, green, line width=1.1pt] (v2) edge (v4); 
                \path[->, orange, line width=0.4pt] (v3) edge (v4); 
                \path[->, orange, line width=0.4pt] (v3) edge (v2); 
                \path[->] (v4) edge (v5); 
                
                \path[->, dashed] (8.8,0) edge (11.3,0);
                \node at (9.9,0.25) {Insplitting};
                
                 \node (v6) at (12,2) {$v_1$};
                \node (v7) at (12,0) {$v_2$};
                \node (v8) at (12,-2) {$v_3$};
                \node (v9) at (16,2) {$w_1$};
                \node (v10) at (20,0) {$v_4$};
                \node (v11) at (16,0) {$w_2$};
                \node (v12) at (16,-2) {$w_3$};
                \path[->, blue, line width=1.7pt] (v6) edge (v9); 
                \path[->, green, line width=1.1pt] (v6) edge (v11); 
                \path[->, bend left = 40, blue, line width=1.7pt] (v6) edge (v10);
                \path[->, green, line width=1.1pt] (v7) edge (v11); 
                \path[->, orange, line width=0.4pt] (v8) edge (v12); 
                \path[->, orange, line width=0.4pt] (v8) edge (v7); 
                \path[->] (v9) edge (v10); 
                \node(v13) at (17.5,1.6) {$e^1$};
                \path[->] (v11) edge (v10); 
                \node(v14) at (17.5,0.25) {$e^2$};
                \path[->] (v12) edge (v10); 
                \node(v15) at (17.5,-1) {$e^3$};
            \end{tikzpicture}
        \end{minipage}
    \end{center}
    \caption[Illustration of the application of move I for hypergraphs.]{Illustration of the application of move I for hypergraphs. Each color/thickness symbolizes one edge and $\mathcal{E}_j$ are one-point sets.}
\end{figure}

\subsection{Move S -- removing a source}

We now begin with proving Thm. \ref{thmmovessummary} step by step, or rather move by move.
For graphs  move S leads to Morita equivalent $C^*$-algebras \cite[Prop. 3.1]{sorensen_2013}, as one can show that $C^*(H\Gamma_S)$ is isomorphic to a full corner of $C^*(H\Gamma)$. For hypergraphs we only get the following proving the statement for move S in Thm. \ref{thmmovessummary}.

\begin{proposition}[Move S] \label{MoveS_homomorphism_specialcase}
    Let $H\Gamma=(V, E, r, s)$ be a finite hypergraph with source $w$ and assume that $H\Gamma$ is locally ultra in $w$. Then there is a *-homomorphism $\pi: C^*(H\Gamma_S)\rightarrow C^*(H\Gamma)$ such that the *-subalgebra $Im(\pi)$ is a full corner in $C^*(H\Gamma)$.
\end{proposition}

\begin{proof}
    We show that $\{p_v,\; s_e\;|\; v \in V_S, \;e\in E_S\}$ is a Cuntz-Krieger $H\Gamma_S$-family in $C^*(H\Gamma)$. The first two hypergraph relations hold in general, even without the restriction on the source $w$ since the move does not change anything at the corresponding edges and vertices. For the third hypergraph relation we note that by the given restriction for each $v \neq w$ it follows that $e \in E_S$ for each edge $e\in E$ with $v\in s(e)$, since $w \notin s(e)$. Thus we get by the third hypergraph relation of $C^*(H\Gamma)$
    \begin{equation*}
        p_v \leq \sum_{e \in E, v \in s(e)}s_es_e^*=\sum_{e \in E_S, v \in s(e)} s_es_e^*\qquad \forall v \in V_S.
    \end{equation*}
    Hence the universal property yields the canonical *-homomorphism $\pi: C^*(H\Gamma_S)\rightarrow C^*(H\Gamma)$ sending the canonical generators $q_v \mapsto p_v$ for all $v \in V_S$ and $t_e \mapsto s_e$ for all $e \in E_S$. 
    We define the projection $p:=\sum_{v\in V_S}p_v$ and claim first that $Im(\pi)=pC^*(H\Gamma)p$. Since
        \begin{align*}
            p_v&=pp_vp && \forall v \in V_S,\\
            s_e&=\big(\sum_{v\in s(e)}p_v\big)s_e\big(\sum_{v \in r(e)}p_v\big)=p\big(\sum_{v\in s(e)}p_v\big)s_e\big(\sum_{v \in r(e)}p_v\big)p && \forall e \in E_S,
        \end{align*}
    the image of the canonical generators is contained in $pC^*(H\Gamma)p$. Hence, $Im(\pi)\subseteq pC^*(H\Gamma)p$. 
    On the other hand it holds for all paths $\mu$ in $H\Gamma$ by Proposition \ref{source_range_projections}
    \begin{itemize}
        \item $ps_{\mu}=\begin{cases} 0 & \text{if} \; s(\mu)=w\\ s_{\mu}\in C^*(p_v,s_e) & \text{else}, \end{cases}$    
        \item $ps_{\mu}^*=s_{\mu}^*$,
        \item $s_{\mu}p=s_{\mu}$,
        \item $s_{\mu}^*p=\begin{cases} 0 & \text{if} \; s(\mu)=w\\ s_{\mu}^*\in C^*(p_v,s_e) & \text{else}. \end{cases}$
    \end{itemize}
    We consider a general element $s_{\mu_1}^{\epsilon_1}...s_{\mu_n}^{\epsilon_n}\neq 0$ with paths $\mu_j$ in $H\Gamma$, $\epsilon_j\in \{1,*\}$ and $\epsilon_j\neq \epsilon_{j+1}$. Since $w$ is a source, 
    we get that only the first and last isometries in $s_{\mu_1}^{\epsilon_1}...s_{\mu_n}^{\epsilon_n}$ can correspond to edges with source $w$. Thus we can use the relations above and get that $ps_{\mu_1}^{\epsilon_1}...s_{\mu_n}^{\epsilon_n}p \in span\{p_v,\;s_e|v\in V_S, \;e\in E_S\}\subseteq Im(\pi)$. Hence $pC^*(H\Gamma)p\subseteq Im(\pi)$. Combining both parts we get the claimed equality. 
    
    To show that the corner $pC^*(H\Gamma)p$ is full, let $I$ be a closed two-sided ideal containing the corner. Thus $I$ contains $\{p_v, s_e \;|\;v \in V_S, e \in E_S\}$ by definition of $p$ and Proposition \ref{generalrelations}. Then we note, that for all $e \in E$ with $s(e)=w$ we have $p_{r(e)}\in I$ and hence $s_e=s_ep_{r(e)} \in I$.
    Given our special case, we get by combining the second and third hypergraph relation, that $p_w=\sum_{e \in E, s(e)=w}s_es_e^*\in I$, as a linear combination of elements in the ideal. Hence, $I$ contains all generators of $C^*(H\Gamma)$ and must thus be equal to it.
\end{proof}

\subsection{Move R -- reduction at a non-sink}

For graphs the definition of the edge $e_f$ in move R just yields the path $ef$. To simplify the upcoming proof, we restrict ourselves to moves R at particular vertices $w$ such that $e_f=ef$. This is not really a restriction, since we can transform any finite hypergraph into a hypergraph with this condition using the decomposition of ranges in Theorem \ref{rangesingularization}. 

\begin{proposition}[Move R]
    Let $H\Gamma=(V, E, r, s)$ be a finite hypergraph with vertex $w \in V$ that emits exactly one edge $f$ and only one vertex $x$ emits to $w$. 
    Then there is a *-homomorphism $\pi: C^*(H\Gamma_R)\rightarrow C^*(H\Gamma)$ such that the *-subalgebra $Im(\pi)$ is a full corner in $C^*(H\Gamma)$.
\end{proposition}

\begin{proof}
    The elements $\{Q_v\;|\; v\in V_R\}$ and $\{T_y \;|\; y \in E_R\}$ defined as 
        \begin{align*}
            &Q_v:=p_v, \\
            &T_y:=\begin{cases} s_e & \text{if}\; y=e\in E\setminus \left(r^{-1}(\{w\}) \cup \{f\}\right) \\ s_{ef} & \text{if}\; y=e_f \in \left\{e_f\;|\; e \in E, r(e)=\{w\}\right\}\end{cases} 
        \end{align*}
    form a Cuntz-Krieger $H\Gamma_R$-family in $C^*(H\Gamma)$. Indeed, the elements $Q_v$ are clearly mutually orthogonal projections and the elements $T_y$ are partial isometries since $ef$ are a quasi perfect paths.
    The first hypergraph relation for $e\in E_R$ follows directly from the first hypergraph relation for $H\Gamma$, since $w\notin r(e)$ implies that the range is completely contained in $V_R$. For $e_f\in E_R$ we have 
        \begin{align*}              T_{e_f}^*T_{e_f}=s_{ef}^*s_{ef}=s_f^*s_e^*s_es_f=s_f^*p_{r(e)}s_f=s_f^*p_wp_{r(e)}s_f=s_f^*p_ws_f=s_f^*s_f=p_{r(f)}=Q_{r_R(e_f)}.
        \end{align*}
    The condition $T_y^*T_z=0$ for $y\neq z$ follows directly from $s_e^*s_g=0$ for $e\neq g$. The second hypergraph relation is again clear for $e \in E_R$. For $e_f\in E_R$ we get 
        \begin{align*}
            T_{e_f}T_{e_f}^* = s_{ef}s_{ef}^*=s_es_fs_f^*s_e^*\leq s_ep_{s(f)}s_e^*=s_ep_wp_{s(f)}s_e^*=s_es_e^*\leq p_{s(e)}=Q_{s_R(e_f)}
        \end{align*}
    while we used in the last step that $w \notin s(e)$ since $e\neq f$. It remains to check the third hypergraph relation. We have for all vertices $v \in V_R$
        \begin{align*}
            Q_v&=p_v\\
            &\leq\sum_{e \in E, v \in s(e)}s_es_e^*\\
            &=\sum_{e \in E, v \in s(e), w \notin r(e)}s_es_e^*+\sum_{e \in E, v \in s(e), w \in r(e)}s_es_e^*.
    \intertext{At this stage we need the restriction that $s(f)=w$ to get $s_fs_f^*=p_w$ and thus by Lemma \ref{source_range_projections} that $s_es_e^*=s_es_fs_f^*s_f$. With this we get}
            &=\sum_{e \in E, v \in s(e), w \notin r(e)}T_eT_e^*+\sum_{e \in E, v \in s(e), w \in r(e)}T_{e_f}T_{e_f}^*\\
            &=\sum_{y \in E_R, v \in s_R(y)}T_yT_y^*.
        \end{align*}
    Thus, we get by the universal property the *-homomorphism  $\pi: C^*(H\Gamma_R)\rightarrow C^*(H\Gamma)$ which maps the generators $q_v \mapsto Q_v$ for all $v \in V_R$ and $t_y \mapsto T_y$ for all $y \in E_R$.

    We define the projection $p:=\sum_{v \in V_R}p_v$. 
    Then
        \begin{align*}
            Q_v&=p_v=pp_vp,\\
            T_e&=s_e=p_{s(e)}s_ep_{r(e)}=pp_{s(e)}s_ep_{r(e)}p, \\
            T_{ef}&=s_es_f=p_{s(e)}s_es_fp_{r(f)}=pp_{s(e)}s_es_fp_{r(f)}p,
        \end{align*}
    where we used that $w \notin s(e)$ for $e \neq f$, $w \notin r(e)$ for $e \in E\setminus \left(r^{-1}(\{w\}) \cup \{f\}\right)$ and that $w \notin r(f)$ as $s(f)=w \neq x$. Thus $Im(\pi)$ is a subset of the corner $pC^*(H\Gamma)p$.
    
    To show that the corner is contained in $Im(\pi)$ we first consider some properties of the crucial edges in $r^{-1}(\{w\}) \cup \{f\}$ and the interaction of the corresponding partial isometries. Let $\mu=\mu_1\dots \mu_n$ be a path in $H\Gamma$. We then have:
    
    \begin{itemize}
        \item[(i)] If $\mu_1=f$, it holds by Lemma \ref{source_range_projections} that $ps_\mu=pp_ws_\mu=0$ and similarly $s_\mu^*p=0$. The first hypergraph relation  gives that $s_f^*s_e=s_e^*s_f=\delta_{e,f}p_{r(f)}=Q_{r(f)}$ for all $e \in E$.
        \item[(ii)] If $\mu_j=f$, for $j=2,\dots n$, by the definition of paths we must have $\mu_{j-1}\in r^{-1}(\{w\})$ and hence $s_{\mu_{j-1}}s_{\mu_j}=T_{{\mu_{j-1}}_f}$.
        \item[(iii)] If $\mu_j \in r^{-1}(\{w\})$, for $j=1,\dots, n-1$, we get again by the definition of paths that $\mu_{j+1}=f$ and hence $s_{\mu_{j}}s_{\mu_{j+1}}=T_{{\mu_j}_f}$.
        \item[(iv)] If $\mu_n \in r^{-1}(\{w\})$, we get by Lemma \ref{source_range_projections} that $s_\mu p=s_\mu p_w p=0$ and similarly $p s_\mu^*=0$. Furthermore, by definition of paths and the fact that only one vertex emits to $w$, $s_{\mu_n} s_e^*\neq 0$ if and only if $e \in r^{-1}(\{w\})$. Since in our special case $s_fs_f^*=p_w$ and $r(\mu_n)=\{w\}$ we get for $e \in r^{-1}(\{w\})$ again by Lemma \ref{source_range_projections} that $s_{\mu_n} s_e^*=s_{\mu_n}p_w s_e^*=s_{\mu_n} s_fs_f^*s_e^*=T_{{\mu_n}_f}T_{e_f}^*$. Similarly we can show that $s_es_f^*=T_{e_f}T_{{\mu_n}_f}^*$.
    \end{itemize}
    
Combining these properties, we get $pSp \in Im(\pi)$, for a general element $S:=s_{\mu_1}^{\epsilon_1}\dots s_{\mu_n}^{\epsilon_n} \in C^*(H\Gamma)$ where $\mu_1,\dots, \mu_n$ are paths in $H\Gamma$ and $\epsilon_j\in \{1,*\}$, $\epsilon_j\neq \epsilon_{j+1}$. Prop. \ref{Hypergraph_Def} then yields the claim. \\

It remains to show that the corner is full. Let $I$ be a closed two-sided ideal containing the corner $pC^*(H\Gamma)p$. Then $I$ contains all projections corresponding to the vertices in $V_R$. Consider $e\in r^{-1}(\{w\})$. Then $e \neq f$ and $w \notin s(e)$ and hence $p_{s(e)} \in I$. Thus $s_e=p_{s(e)}s_e \in I$ by properties of the ideal. The first Cuntz-Krieger relation then gives $p_w=s_e^*s_e\in I$. Hence by Proposition \ref{unit}, the ideal contains the unit and hence it must be all of $C^*(H\Gamma)$. Thus the corner is not contained in a proper closed two sided ideal and is thus full. 
\end{proof}

\subsection{Move O -- outsplitting}

\begin{proposition}[Move O]
    Let $H\Gamma=(V, E, r,s)$ be a finite hypergraph and $w$ be a vertex that is not a sink and let $H\Gamma$ be locally ultra at $w$. Let $H\Gamma_O$ be the hypergraph obtained by outsplitting $H\Gamma$ at $w$. Then $C^*(H\Gamma) \cong C^*(H\Gamma_O)$.
\end{proposition}

\begin{proof}
    Let $\{q_v \;|\ v \in V_O\}, \; \{t_e \; | \; e \in E_O\}$ be the universal Cuntz-Krieger $H\Gamma_O$-family. We define a Cuntz-Krieger $H\Gamma$-family in $H\Gamma_O$ by
    \begin{align*}
        P_v&:=\begin{cases} q_v & \quad \text{if} \; v \neq w \\ \sum_{i=1}^n q_{w^i} & \quad\text{if} \; v = w,\end{cases}\\
        S_e&:=\begin{cases} t_e & \quad \text{if} \; w \notin r(e) \\ \sum_{i=1}^n t_{e^i} & \quad \text{if} \; w \in r(e)\end{cases}.
    \end{align*}
    
    Indeed, the elements $P_v$ are mutually orthogonal projections and the elements $S_e$ are clearly partial isometries if $w \notin r(e)$. For the other case we note that the ranges of $e^1,\dots , e^n$ are disjoint, to get the required result. 
    
    For the first hypergraph relation the case $w \notin r(e)$ is straightforward. For $w \in r(e)$ we get:
        \begin{align*}
            S_e^*S_e
            =\sum_{i=1}^nt_{e^i}^*t_{e^i}
            =\sum_{i=1}^nq_{r_O(e^i)}=q_{r(e)\setminus w}+ \sum_{i=1}^nq_{w^i}=P_{r(e)\setminus w}+P_w
            =P_{r(e)}.
        \end{align*}
    By the hypergraph relations of $H\Gamma_O$ we know that $t_{e^i}^*t_{f^j}=0$ for $e \neq f$ or $i \neq j$, which implies $S_e^*S_f=0$ for $e \neq f$. 
    
    For the second hypergraph relation we get for $w \notin r(e)$: 
        \begin{align*}
            S_eS_e^*
            =t_{e^1}t_{e^1}^*\leq q_{s_O(e^1)}
            =\begin{cases} P_{s(e)} & e \notin \mathcal{E}_j\\q_{s(e)\setminus w} + q_{w^j}\leq q_{s(e)\setminus w} +\sum_{j=1}^nq_{w^j}
            =P_{s(e)}& e \in \mathcal{E}_j.\end{cases}
        \end{align*}
    For $w \in r(e)$ it follows using the first equation:
        \begin{align*}
            S_eS_e^*
            =\sum_{i=1}^nt_{e^i}t_{e^i}^*
            \leq\sum_{i=1}^nq_{s_O(e^i)}.
        \end{align*}
    Similar as in the equation above we get $q_{s_O(e^i)}\leq P_{s(e)}$ for all $i=1,\dots n$. Using the definition of the order it follows that $\sum_{i=1}^nq_{s_O(e^i)}\leq P_{s(e)}$ which yields the required result. 
    
    We finally tackle the last hypergraph relation. For $v \neq w$ we have 
        \begin{align*}
            P_v &=q_v\\
            &\leq \sum_{x \in V_O, v \in s_O(x)}t_xt_x^*\\
            &= \sum_{e \in E, v \in s(e), w \notin r(e)}t_{e^1}t_{e^1}^*+\sum_{e \in E, v \in s(e), w\in r(e)}\sum_{i=1}^nt_{e^i}t_{e^i}^*\\
            &=\sum_{e \in E, v \in s(e), w \notin r(e)}t_{e^1}t_{e^1}^*+\sum_{e \in E, v \in s(e), w\in r(e)} \left(\sum_{i=1}^nt_{e^i}\right) \left(\sum_{i=1}^nt_{e^i}\right)^*\\
            &=\sum_{e \in E, v \in s(e), w \notin r(e)}S_{e}S_{e}^*+\sum_{e \in E, v \in s(e), w\in r(e)}S_{e}S_{e}^*\\
            &=\sum_{e \in E, v \in s(e)}S_{e}S_{e}^*.
        \end{align*}
    For $v=w$ we can duplicate the above calculation to get for each $i=1,\dots,n$ that 
        \begin{align*}
            q_{w^i}\leq \sum_{e \in \mathcal{E}_i}S_{e}S_{e}^*.
        \end{align*}
    Using that $\{e \in E\; |\; w \in s(e)\}=\mathcal{E}_1\cup \dots \cup \mathcal{E}_n$ and $S_eS_e^*$ are mutually orthogonal projections we get 
        \begin{align*}
            P_w=\sum_{i=1}^nq_{w^i}\leq \sum_{e \in E, w \in s(e)}S_{e}S_{e}^*,
        \end{align*}
    which completes the proof of the Cuntz-Krieger family.

    To obtain a Cuntz-Krieger $H\Gamma_O$-family in $H\Gamma$ we define
    \begin{align*}
        Q_v&:=\begin{cases} p_v & \quad \text{if} \; v \neq w^j \\ \sum_{e \in \mathcal{E}_j} s_es_e^* & \quad\text{if} \; v = w^j,\end{cases}\\
        T_{e^i}&:=\begin{cases} s_e & \qquad \;\text{if} \; w^j \notin r_O(e^i) \\ s_eQ_{r_O(e^i)} & \qquad \;\text{if} \; w^j \in r_O(e^i).\end{cases}
    \end{align*}

    Since we assumed that for all $e \in E$ with $w \in s(e)$ it follows $w=s(e)$, it follows for all vertices $v\neq w$ that $(s_es_e^*)p_v=0=p_v(s_es_e^*)$. Hence, since the sets $\mathcal{E}_j$ are disjoint and the projections $p_v$ are mutually orthogonal, we get using the first hypergraph relation of $C^*(H\Gamma)$ that the projections $Q_v$ are mutually orthogonal. Furthermore we get that $Q_{w^i}p_w=Q_{w^i}$, which will be useful later on. 
    
    The first relation for $e^1 \in E_O$ with $w\notin r(e)$ is obvious. In case that $w \in r(e)$ we get for $e^1$
        \begin{align*}
            T_{e^1}^*T_{e^1}&=Q_{r_O(e^1)}s_e^*s_eQ_{r_O(e^1)}\\
            &=Q_{r_O(e^1)}p_{r(e)}Q_{r_O(e^1)}\\
            &=\left(Q_{w^1}+p_{r(e)\setminus\{w\}}\right)p_{r(e)}\left(Q_{w^1}+p_{r(e)\setminus\{w\}}\right)\\
            &= Q_{w^1}+p_{r(e)\setminus\{w\}}\\
            &=Q_{r_O(e^1)},
        \end{align*}
    using that $Q_{w^j}p_v=\delta_{v,w}Q_{w^j}$. The case for $i=2,\dots,n$ follows similarly using $Q_{r_O(e^i)}=Q_{w^i}$.
    By the first hypergraph relation for $C^*(H\Gamma)$ and the orthogonality of the projections we get $T_{e^i}^*T_{f^j}=0$ for $e^i\neq f^j$. Furthermore, using these results it is straightforward to see that the elements $T_{e^i}$ are partial isometries.
    
    For the second hypergraph relation we again consider the case $i=1$ and $w \in r(e)$ first. We have
        \begin{align*}
            T_{e^1}T_{e^1}^*=s_es_e^*\leq \begin{cases} p_{s(e)}=Q_{s_O(e^1)} & \text{if} \; s(e)\neq w\\
            \sum_{f \in \mathcal{E}_j}s_fs_f^*=Q_{w^j}=Q_{s_O(e^1)}& \text{if} \; s(e)= w,\end{cases}
        \end{align*}
    where we used the assumption that either $w \notin s(e)$ or $w=s(e)$ for all $e \in E$ and that the elements $s_fs_f^*$ are mutually orthogonal projections. For $w =s(e)$ we have using $Q_{r_O(e^i)}\leq 1$
        \begin{align*}
            T_{e^i}T_{e^1}^*=s_eQ_{r_O(e^i)}s_e^*\leq s_es_e^*
        \end{align*}
    which can be estimated similar to the previous case.
    
    Finally we check the third hypergraph relation. We note, that the assumption that $w \in s(e)$ implies $w =s(e)$ leads to $p_w=\sum_{j=1}^nQ_{w^j}$. For $v \in V\setminus \{w\}$ we then have using the third hypergraph relation for $C^*(H\Gamma)$
        \begin{align*}
            Q_v&=p_v\\
            &\leq \sum_{e \in E, v \in s(e)}s_es_e^*\\
            &=\sum_{e \in E, v\in s(e), w\notin r(e)}T_{e^1}T_{e^1}^* &&+ \sum_{e \in E, v\in s(e), w\in r(e)}s_ep_{r(e)}s_e^*\\
            &=\sum_{e \in E, v\in s(e), w\notin r(e)}T_{e^1}T_{e^1}^* &&+ \sum_{e \in E, v\in s(e), w\in r(e)}s_e\left(\sum_{j=1}^nQ_{w^j}+p_{r(e)\setminus \{w\}}\right)s_e^*\\
            &=\sum_{e \in E, v\in s(e), w\notin r(e)}T_{e^1}T_{e^1}^* &&+ \sum_{e \in E, v\in s(e), w\in r(e)}\left(s_e(Q_{w^1}+p_{r(e)\setminus \{w\}})s_e^*+\sum_{j=2}^ns_eQ_{w^j}s_e^* \right)\\
            &=\sum_{e \in E, v\in s(e), w\notin r(e)}T_{e^1}T_{e^1}^* &&+ \sum_{e \in E, v\in s(e), w\in r(e)}\sum_{j=1}^nT_{e^j}T_{e^j}^*\\
            &=\sum_{e^i \in E_O, v \in s_O(e^i)}T_{e^i}T_{e^i}^*.
        \end{align*}
    For $v = w^j$ we have by definition that $Q_{w^j}=\sum_{e \in E, w^j \in s(e)}s_es_e^*$. Hence the same calculation as above yields the required result.

    Using the above Cuntz-Krieger families we get the canonical *-homomorphism 
        \begin{align*}
            \pi&: C^*(H\Gamma) \rightarrow C^*(H\Gamma_O), \quad p_v \mapsto P_v, \quad s_e \mapsto S_e,\\
            \tilde{\pi}&: C^*(H\Gamma_O) \rightarrow C^*(H\Gamma), \quad q_v \mapsto Q_v, \quad t_e \mapsto T_e.
        \end{align*}
    Straightforward calculations show that both *-homomorphisms are inverse to each other on the generators. Thus they are inverse on the whole $C^*$-algebras and we get the required isomorphism. 
\end{proof}

\subsection{Move I -- insplitting}

Before taking care of move I, we have a look at the indelay which introduces vertices to delay the arrival of an edge on its range. One can define this even in a more general setting with a so called Drinen range vector as done in \cite[Ch. 4]{flowequivalence}. We only consider the special case needed for the connection to move I, which we consider afterwards. The constructions and proofs in the following section are adapted from \cite[Ch. 4 and 5]{flowequivalence} and extended to the hypergraph setting. The upcoming proofs are again quite technical and deal with similar case distinctions as seen for move O. We will thus only highlight the critical steps.

\begin{definition}[\textbf{Indelay}]
    Let $H\Gamma=(V, E, r,s)$ be a finite hypergraph and $w$ be a vertex that is not a source. We partition the set of incoming edges in finitely many nonempty sets: 
    \begin{equation*}
        \{e \in E\; |\; w \in r(e)\}=\mathcal{E}_1 \cup \dots \cup \mathcal{E}_n.
    \end{equation*}
    The hypergraph $H\Gamma_{D}$ obtained by an \emph{indelay}\index{Indelay} of $H\Gamma$ at $w$ is defined by 
    \begin{align*}
        &V_{D}:=V\setminus \{w\} \cup \{w^1,\dots,w^n\},\\
        &E_D:=E \cup \{f^1,\dots f^n\},\\
        &r_{D}(e):=\begin{cases} r(e) &\quad \text{if} \;w \notin r(e)\\
        \left(r(e)\setminus \{w\}\right) \cup \{w^j\} &\quad \text{if} \; w \in r(e), \end{cases}\\ 
        &r_{D}(f^j):=w^j,\\
        &s_{D}(e):=\begin{cases} s(e) & \quad \text{if} \; w \notin s(e) \\ \left(s(e)\setminus \{w\}\right) \cup \{w^1\} &\quad \text{if} \; w \in s(e),\end{cases} \\
        &s_{D}(f^j):=w^{j+1}.
    \end{align*}
\end{definition}

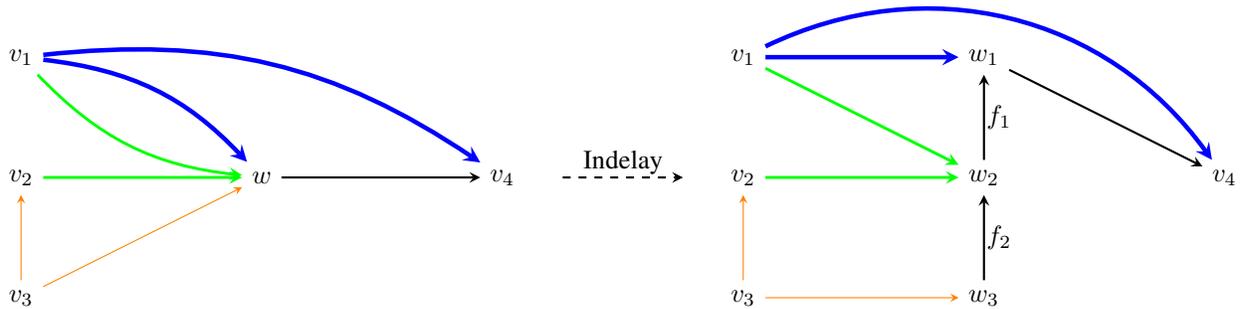
\begin{figure}[h!]
        \centering
             \begin{center}
        \begin{minipage}{\linewidth}
        \centering
            \begin{tikzpicture}[ > = stealth, auto, thick, scale=0.8]
                \node (v1) at (0,2) {$v_1$};
                \node (v2) at (0,0) {$v_2$};
                \node (v3) at (0,-2) {$v_3$};
                \node (v4) at (4,0) {$w$};
                \node (v5) at (8,0) {$v_4$};
                \path[->, bend left = 20, blue, line width=1.7pt] (v1) edge (v4); 
                \path[->, bend right=20, green, line width=1.1pt] (v1) edge (v4); 
                \path[->, bend left = 20, blue, line width=1.7pt] (v1) edge (v5); 
                \path[->, green, line width=1.1pt] (v2) edge (v4); 
                \path[->, orange, line width=0.4pt] (v3) edge (v4); 
                \path[->, orange, line width=0.4pt] (v3) edge (v2); 
                \path[->] (v4) edge (v5); 
                
                \path[->, dashed] (9,0) edge (11,0);
                \node at (10,0.25) {Indelay};
                
                \node (v6) at (12,2) {$v_1$};
                \node (v7) at (12,0) {$v_2$};
                \node (v8) at (12,-2) {$v_3$};
                \node (v9) at (16,2) {$w_1$};
                \node (v10) at (20,0) {$v_4$};
                \node (v11) at (16,0) {$w_2$};
                \node (v12) at (16,-2) {$w_3$};
                \path[->, blue, line width=1.7pt] (v6) edge (v9); 
                \path[->, green, line width=1.1pt] (v6) edge (v11); 
                \path[->, bend left = 40, blue, line width=1.7pt] (v6) edge (v10);
                \path[->, green, line width=1.1pt] (v7) edge (v11); 
                \path[->, orange, line width=0.4pt] (v8) edge (v12); 
                \path[->, orange, line width=0.4pt] (v8) edge (v7); 
                \path[->] (v9) edge (v10); 
                \path[->] (v11) edge (v9); 
                \node(v13) at (16.25,1) {$f_1$};
                \path[->] (v12) edge (v11); 
                \node(v14) at (16.25,-1) {$f_2$};
            \end{tikzpicture}
        \end{minipage}
    \end{center}
        \caption[Illustration of the application of the indelay.]{Illustration of the application of the indelay. The colored edges of different thickness symbolize one edge each. The incoming edges are partitioned into one-point sets.}
    \end{figure}

The next proposition uses some kind of locally ultra property at the range rather than at the source.

\begin{proposition}\label{surjective_hom_indelay} 
    Let $H\Gamma=(V, E, r,s)$ be a finite hypergraph, $w$ be a vertex that is not a source such that $w \in r(e)$ implies $r(e)=\{w\}$. $H\Gamma_{D}$ be the hypergraph obtained by an indelay of $H\Gamma$ at $w$.  Then there is a surjective *-homomorphism from $C^*(H\Gamma_D)$ onto a full corner of $C^*(H\Gamma)$.
\end{proposition}

\begin{proof}
    Let $\{q_v, t_e \}$ be the universal Cuntz-Krieger $H\Gamma_{D}$-family. Then $\{P_v \;|\ v \in V\}, \; \{S_e \; | \; e \in E\}$ defined as
    \begin{align*}
        P_v&:=\begin{cases} q_v & \quad \text{if} \; v \neq w \\ q_{w^1} & \quad\text{if} \; v = w\end{cases}\\
        S_e&:=\begin{cases} t_e & \quad \text{if} \; w \notin r(e) \\ t_{e}t_{f_{j-1}}\dots t_{f_1} & \quad \text{if} \; w \in \mathcal{E}^j\end{cases}
    \end{align*}
    forms a Cuntz-Krieger $H\Gamma$-family in $H\Gamma_{D}$.
    The hypergraph relations follow by mainly straightforward calculations. We mention shortly the critical tricks. For the first hypergraph relations note that $f_j\dots f_1$ are perfect paths for all $j \in \{1,\dots n-1\}$. Thus, for all $e \in \mathcal{E}_j$ it follows that
        \begin{align*}
            S_e^*S_e= q_{r_D(f_1)}=q_{w_1}=P_w. 
        \end{align*}
    This result also explains, why we have to add the assumption that $w \in r(e)$ implies $w=r(e)$. The assumption implies furthermore, that $ef_j\dots f_1$ is a perfect path and since $q_{w_j}=t_{f_{j-1}}t_{f_{j-1}}^*$ by the construction of the edges $f_j \in E_D$ this leads to $S_eS_e^*=t_et_e^*$ for all $e \in E$, which is crucial for the second hypergraph relation. For the third hypergraph relation we use that $v \in s_D(e)$ implies $v \in s(e)$ and $w_1 \in s_D(e)$ implies $w \in s(e)$.
    Combining this with the hypergraph relations of $H\Gamma_D$ shows that the given elements form a Cuntz-Krieger $H\Gamma$-family.

    By the universal property we get the canonical *-homomorphism  $\pi:C^*(H\Gamma)\rightarrow C^*(H\Gamma_D)$, $p_v \mapsto P_v$, $s_e\mapsto S_e$.
    Let $F:=V\setminus \{w\} \cup \{w^1\}\subseteq V_D$ and $p:=q_F$.
    By definition of $P_v$ we get $P_v=pP_vp$. For all $e\in E$ we have $s_D(e)\subseteq F$. For $e \in E$ with $w \notin r(e)$ we have $r_D(e)\subseteq V\setminus \{w\}\subseteq F$ and for $e \in \mathcal{E}_j$ we have $r_D(ef_{j-1}\dots f_1)=w_1\subseteq F$. Hence by applying Proposition \ref{rangeprojectionquasiperfectpaths} we get $S_e=pS_ep$. Thus, the image of the generators of $C^*(H\Gamma)$ is contained in $pC^*(H\Gamma_D)p$ and $Im(\pi)\subseteq pC^*(H\Gamma_D)p$. \\
    
    To see the converse we consider a general element $S:=t_{\mu_1}^{\epsilon_1}\dots t_{\mu_m}^{\epsilon_m}$ for paths $\mu_1, \dots, \mu_m$ in $H\Gamma_D$ and $\epsilon_1, \dots, \epsilon_m \in \{1,*\}$. We show that $pSp\in Im(\pi)$. We first have a deeper look at a path $\mu=e_1\dots e_k$ in $H\Gamma_D$. 
        \begin{itemize}
            \item[(i)] If $e_j \in E$ for $j=1, \dots, k$, we have $s_D(\mu)\subseteq F$. Hence $pt_\mu=t_\mu=S_\mu \in Im(\pi)$.
            \item[(ii)] If $e_j \in E$ for $j=1, \dots, k$ and $w_i \notin r_D(\mu)$, we have $r_D(\mu)\subseteq V\setminus\{w\}$. Hence $t_\mu p=t_\mu=S_\mu \in Im(\pi)$. On the other hand, if $w_i\in r_D(\mu)$ we have $w_i =r_D(\mu)$ by assumption and hence $t_\mu p=0\in Im(\pi)$.
            \item[(iii)] If $e_j=f_l$ for some $j\in \{ 1, \dots, k\}$ and $l\in\{1,\dots,n-1\}$, either the whole path $f_l\dots f_{n-1}$ is contained in $\mu$ or $e_k=f_{l+k-j}$. If $e_1\neq f^l$ and $e_k\neq f_{l+k-j}$, the path $\nu:=e_{j-1}f_l\dots f_{n-1}$ is contained in $\mu$ and hence $t_\mu=s_{e_1}\dots S_\nu \dots s_{e_k}\in Im(\pi)$. Hence it remains to consider the cases, when the path starts or ends with an element in $\{f_1,\dots, f_{n-1}\}$.
                \begin{itemize}
                    \item[(iii.a)] If $e_1=f^l$ we have $s_D(\mu)=w_l\notin F$ and hence $pt_\mu=0\in Im(\pi)$. On the other hand, if we have a second path $\alpha$ such that $t_\alpha^*t_\mu\neq 0$ or $t\mu^*t_\nu\neq 0$, we must have $\alpha_j=f_{l+j-1}$ since we have perfect paths. Hence using properties of perfect paths the elements vanish and we are left with paths $\mu', \alpha'$ which dies not contain the elements $\{f_1, \dots, f_n\}$. Hence $t_{\mu'}=s_{\mu'}\in Im(\pi)$.
                    \item[(iii.b)] If $e_k=f_{l+k-j}$ we have $t_\mu p=0\in Im(\pi)$. A similar argument as in the last step shows, that the interaction with another path $\alpha$ cancels the elements $t_{f_j}$ and we are left with $t_{\mu'}=s_{\mu'}\in Im(\pi)$.
                \end{itemize}
        \end{itemize}
    Combining these arguments it follows that $pSp\in Im(\pi)$ and hence $pSp = Im(\pi)$.\\
    
    It remains to show that the corner is full. Let $I\subseteq C^*(H\Gamma_D)$ be a closed two-sided ideal containing $pC^*(H\Gamma_D)p$. Then I contains the projections $q_v=pq_vp$ for all $v \in V\setminus\{w\}\cup \{w_1\}$. Since $s_D(e)\subseteq V\setminus \{w\} \cup \{w^1\}$ for all $e \in E$, we get $t_e=pt_e\in I$ for all $e \in E$. Since $\mathcal{E}_j\neq \emptyset$, for each vertex $w_j$ there exists an edge $e \in E$ such that $w_j=r_D(e)$. Hence since $p_{w_j}=t_e^*t_e\in I$. Thus all canonical projections are contained in the ideal, and thus by Proposition \ref{unit} the unit is contained in the ideal. This shows that the ideal must be all of $C^*(H\Gamma_D)$ and the corner is full. 
\end{proof}

Now we can connect the indelay with move I and receive isomorphic $C^*$-algebras.

\begin{proposition}\label{isomorphism_indelay_MoveI}
    Let $H\Gamma=(V, E, r,s)$ be a finite hypergraph and $w$ be a vertex that is not a source and let $H\Gamma$ be ultra locally at $w$. The incoming edges of $w$ be partitioned into disjoint sets $\mathcal{E}_1 \cup \dots \cup \mathcal{E}_n$. Let $H\Gamma_{D}$ and $H\Gamma_I$ be the corresponding hypergraphs formed by an indelay and an insplitting respectively. Then $C^*(H\Gamma_{D})\cong C^*(H\Gamma_I)$. 
\end{proposition}

\begin{proof}
Let $\{p_v, s_e\}$ and $\{q_v, t_e\}$ be the canonical generators of $C^*(H\Gamma_I)$ and $C^*(H\Gamma_D)$ respectively. We define a Cuntz-Krieger $H\Gamma_I$-family in $C^*(H\Gamma_D)$ by 
        \begin{align*}
            &P_v:=q_v,\\
            &S_{e^i}:=\begin{cases}t_e & \text{if}\;i=1\\ 
            t_{f_{i-1}}\dots t_{f_1}t_e& \text{if}\; i=2,\dots, n.\end{cases}
        \end{align*}

Indeed, for the first hypergraph relation, we note that by definition, $r_D(e)=r_I(e^i)$ for all $e \in E$. Thus we get for $i=1$ that $S_{e^1}^*S_{e^1}=P_{r_I(e^1)}$.
For $i=2,\dots,n$ we get using the fact that the $f_j$ build perfect paths
    \begin{align*}
        S_{e^i}^*S_{e^i}= t_e^*t_{f_1}^*\dots t_{f_{i-1}}^*t_{f_{i-1}}\dots t_{f_1}t_e=t_e^*t_e=q_{r_D(e)}=P_{r_I(e^i)}.
    \end{align*}
It follows directly that for $e^i\neq g^j$ we have $S_{e^i}^*S_{g^j}=0$.

The second hypergraph relation for $i=1$ we follows using that $s_D(e)=s_I(e^1)$ for all $e\in E$
    \begin{align*}
         S_{e^1}S_{e^1}^*=t_et_e^*\leq q_{s_D(e)}=P_{s_I(e^1)}.
    \end{align*}
For $i=2,\dots,n$ we get, using again that we deal with perfect paths
    \begin{align*}
         S_{e^i}S_{e^i}^*&=t_{f_{i-1}}\dots t_{f_1}t_et_e^*t_{f_1}^*\dots t_{f_{i-1}}^*
        \leq t_{f_{i-1}}^*t_{f_{i-1}}
        = q_{w^i}
        =P_{w^i}
        =P_{s_I(e^i)}.
    \end{align*}   
    
To check the third hypergraph relation, consider first $v \neq w^i$ for $i>1$. Then we have that $v \in s_D(e)$ if and only if $v \in s_I(e^1)$. Recalling that $S_{e^1}=t_e$ we get 
    \begin{align*}
        P_v=q_v 
        \leq \sum_{e \in E_D, v \in s_D(e)}t_et_e^*
        \leq \sum_{e \in E_I, v \in s_I(e)}S_eS_e^*.
    \end{align*}
In case of $v = w^i$ for $i>1$ we get the result using the perfect paths 
    \begin{align*}
        P_{w^i}&=q_{w^i}\\
        &=t_{f_{i-1}}t_{f_{i-1}}^*\\
        &=t_{f_{i-1}}\dots t_{f_1}t_{f_1}^*\dots t_{f_{i-1}}^*\\
        &=t_{f_{i-1}}\dots t_{f_1}q_{w_1}t_{f_1}^*\dots t_{f_{i-1}}^*\\
        &\leq t_{f_{i-1}}\dots t_{f_1}\left(\sum_{e \in E_D, v \in s_D(e)}t_et_e^*\right)t_{f_1}^*\dots t_{f_{i-1}}^*\\
        &=\sum_{e \in E_D, v \in s_D(e)}t_{f_{i-1}}\dots t_{f_1}t_et_e^*t_{f_1}^*\dots t_{f_{i-1}}^*\\
        &=\sum_{e^i \in E_D, w^i \in s_I(e^i)}S_{e^i}S_{e^i}^*.\\        
    \end{align*}
       
On the other hand we can define a Cuntz-Krieger $H\Gamma_D$-family in $C^*(H\Gamma_I)$ by
        \begin{align*}
            Q_v&:=p_v \\
            T_e&:=s_{e^1} \\
            T_{f_j}&:=\sum_{e \in E, w \in s(e)}s_{e^{j+1}}s_{e^j}^*.
        \end{align*}

 For $e \in E$, the first hypergraph relation follows using  $r_I(e^1)=r_D(e)$. For the remaining edges $\{f_1,\dots, f_{n-1}\}$ it holds using that $r_I(e^{j+1})=r_I(e^j)$ and using that $s_I(e^i)=w^i$ for $i=2,\dots,n$
    \begin{align*}
        T_{f_j}^*T_{f_j}&=\left(\sum_{w \in s(e)}s_{e^j}s_{e^{j+1}}^*\right)\left(\sum_{w \in s(e)}s_{e^{j+1}}s_{e^j}^*\right)\\
        &=\sum_{w \in s(e)}s_{e^j}s_{e^{j+1}}^*s_{e^{j+1}}s_{e^j}^*\\
        &=\sum_{w \in s(e)}s_{e^j}s_{e^j}^*\\
        &=p_{w^j}\\
        &=Q_{w^j}\\
        &=Q_{r_D(f_j)}.
    \end{align*}

For $e \in E$ the second hypergraph relation follows directly since $s_I(e^1)=s_D(e)$.
Similar as for the first hypergraph relation, we get for the remaining edges $\{f_1,\dots, f_{n-1}\}$ using that $s_I(e^i)=w^i$ for $i=2,\dots,n$
    \begin{align*}
        T_{f_j}T_{f_j}^*&=\left(\sum_{w \in s(e)}s_{e^{j+1}}s_{e^j}^*\right)\left(\sum_{w \in s(e)}s_{e^j}s_{e^{j+1}}^*\right)\\
        &=\sum_{w \in s(e)}s_{e^{j+1}}s_{e^j}^*s_{e^j}s_{e^{j+1}}^*\\
        &=\sum_{w \in s(e)}s_{e^{j+1}}s_{e^{j+1}}^*\\
        &=p_{w^{j+1}}\\
        &=Q_{w^{j+1}}\\
        &=Q_{s_D(f_j)}.
    \end{align*}
    
The third hypergraph relation for $v \neq w^j$ follows again since $s_I(e^1)=s_D(e)$
    \begin{align*}
        Q_v=p_v\leq \sum_{e^i \in E_I, v \in s_I(e^i)}s_{e^i}s_{e^i}^*=\sum_{e \in E_D, v \in s_D(e)}T_eT_e^*.
    \end{align*}
The case $v = w^j$ for $j=2,\dots, n$ follows directly from the calculation in (HR2a).

Applying the universal property twice we get the *-homomorphisms
    \begin{align*}
        \pi&:C^*(H\Gamma_I)\rightarrow C^*(H\Gamma_D), p_v \mapsto P_v,  s_{e^i} \mapsto S_{e^i} \\
        \tilde{\pi}&:C^*(H\Gamma_D)\rightarrow C^*(H\Gamma_I), q_v \mapsto Q_v, t_e \mapsto T_e,
\end{align*}
which are inverse to each other. To see this, we show that both are inverse to each other on the generators. This is clear for all projections and edges except of $e^i \in C^*(H\Gamma_I)$ with $i=2,\dots,n$ and $f_j\in C^*(H\Gamma_D)$. For these we get 
    \begin{align*}
        \tilde{\pi}\circ \pi (s_{e^i})&=\tilde{\pi}(t_{f_{i-1}}\dots t_{f_1}t_e)\\
        &=\left(\sum_{g \in E}s_{g^{i}}s_{g^{i-1}}^*\right)\dots \left(\sum_{g \in E}s_{g^2}s_{g^1}^*\right)s_{e^1}\\
        &=s_{e^{i}}s_{e^{i-1}}^*\dots s_{e^2}s_{e^1}^*s_{e^1}\\
        &=s_{e^{i}}.
    \end{align*}
Since $w \in s(e)$ implies $w=s(e)$ we get that $q_{w^1}=\sum_{e \in E, w\in s(e)}t_et_e^*$. Using this it follows 
    \begin{align*}
        \pi \circ \tilde{\pi}(t_{f_i})&=\pi(\sum_{e \in E}s_{e^{i+1}}s_{e^i}^*)\\
        &=\sum_{e \in E, w\in s(e)}t_{f_{i}}\dots t_{f_1}t_et_e^*t_{f_1}^*\dots t_{f_{i-1}}^* \\
        &=t_{f_{i}}\dots t_{f_1}\left(\sum_{e \in E, w\in s(e)}t_et_e^*\right)t_{f_1}^*\dots t_{f_{i-1}}^* \\
        &=t_{f_{i}}\dots t_{f_1}q_{w^1}t_{f_1}^*\dots t_{f_{i-1}}^* \\
        &=t_{f^i},
    \end{align*}
where we used the properties of perfect paths. 
\end{proof}

Combining Proposition \ref{surjective_hom_indelay} and Proposition \ref{isomorphism_indelay_MoveI} we get the desired result under the locally ultra assumption.

\begin{korollar}[Move I]
      Let $H\Gamma=(V, E, r, s)$ be a finite hypergraph and $w$ be a vertex that is not a source and let $H\Gamma$ be locally ultra at $w$.  Let $H\Gamma_I$ be the hypergraph obtained by insplitting $H\Gamma$ at $w$. Then there is a surjective *-homomorphism from $C^*(H\Gamma_I)$ onto a full corner of $C^*(H\Gamma)$.  
\end{korollar}

\section{\textbf{Concluding remarks}} \label{Further_Research}

This is the first article on hypergraph $C^*$-algebras and we have to leave many questions open. See also \cite{Schaefer, faross} for very recent follow up articles on hypergraph $C^*$-algebras. Here is a selection of topics for future research.\\

\textbf{(1)} Recall that $C^*(H\Gamma)$ is isomorphic to its dual graph $C^*$-algebra $C^*(\tilde\Gamma)$ under certain conditions, see Cor. \ref{dualgraphisom}.  Given two hypergraphs with similar dual graphs -- do they share any properties? Are all nuclear hypergraph $C^*$-algebras isomorphic to the $C^*$-algebra of its dual graph? \\

\textbf{(2)} Based on our results for finite hypergraphs a next step could be to investigate infinite hypergraphs. We already shed some light on critical steps in the definition, see Definition \ref{Hypergraph_Def_generalized}. Another indication can be the results on infinite ultragraphs in \cite{Tomforde}.\\

\textbf{(3)} Building on the specific characteristics of hypergraph $C^*$-algebras, one can investigate further implications of the path structure, see Section 3.5.\\

\textbf{(4)}  The topic of non-nuclearity offers a broad field of research questions. One can investigate concrete conditions for nuclearity and try to describe them via properties of the hypergraph, see also \cite{Schaefer}. 
    Within this context one could examine the ideal structure of hypergraph $C^*$-algebras and their relation to saturated and hereditary subgraphs. This is especially interesting as hereditary subalgebras of nuclear $C^*$-algebras are nuclear. Thus, this could be used to further enlarge the number of examples of non-nuclear hypergraph $C^*$-algebras.\\

\textbf{(5)}  Our counterexample in Example \ref{failed_gauge_uniqueness} has shown that a direct generalization of the Gauge-Invariant Uniqueness Theorem is not possible. We only have it for a subclass of hypergraph $C^*$-algebras, see Thm. \ref{Gauge_Uniqueness_hypergraphs}. This raises several new research questions: Can the theorem be generalized with another action? Do the restrictions under which the Gauge-Invariant Uniqueness Theorem holds already describe the ultragraph $C^*$-algebras (see Remark \ref{ultragraphs_restrictions_gauge_uniqueness})? Are there other ways to prove injectivity of representations?

In the realm of the Gauge-Invariant Uniqueness Theorem we touched the dual graph of a hypergraph. We saw that it looses information, especially, its $C^*$-algebra is not isomorphic, not even Morita equivalent to the initial hypergraph $C^*$-algebra. Multiple questions are interesting in this regard: Which information is lost? What have  hypergraphs with the same dual graphs  in common? Are there other constructions/generalizations of the dual graph which give more insights?\\

\textbf{(6)} The manipulation of hypergraphs by moves and the corresponding changes in the associated graphs remain an exciting area of research. Based on our results for the moves S, R, I, O, further investigations can be made. In particular, by constructing counterexamples. Of particular interest is also the observation that the hypergraphs must locally look like ultragraphs in order to apply the moves.  \\

\textbf{(7)} For graph $C^*$-algebras there is already an explicit way to compute the K-groups \cite[Thm. 3.2]{Raeburn2003CuntzKriegerAO}. This result can be applied to ultragraph $C^*$-algebras by Morita equivalence. How about the K-theory of hypergraph $C^*$-algebras?\\

\textbf{(8)} In the Bachelor's thesis of the third author, we have two examples of  hypergraph $C^*$-algebras we don't understand, \cite[Sect. 3.2.4]{zenner}. The first one is given by $V=\{v_1,v_2,v_3,v_4\}$, $E=\{e\}$ and $r(e)=\{v_1,v_2\}$, $s(e)=\{v_3,v_4\}$. It is easy to see that $s_e+s_e^*$ is a unitary, but we have no further description of this $C^*$-algebra.

The second example is given by $V=\{v_1,v_2,v_3,v_4\}$, $E=\{e_1,e_2\}$ and $r(e_1)=\{v_3,v_4\}$, $s(e_1)=\{v_1,v_2\}$, $r(e_2)=\{v_1,v_2\}$, $s(e_2)=\{v_3,v_4\}$.  Here, $u:=s_{e_1}+s_{e_2}$ is a unitary satisfying with the projection $p:=s_{e_1}^*s_{e_1}$ the relation $up=(1-p)u$, such that we obtain a (non-surjective) *-homomorphism from $C^*(\mathbb Z/\mathbb Z_2)\rtimes_\alpha\mathbb Z/\mathbb Z_2$ to $C^*(H\Gamma)$, where $\alpha$ is the automorphism flipping the projections $p$ and $1-p$ in $C^*(\mathbb Z/\mathbb Z_2)\cong C^*(p,1)$. It would be interesting to understand $C^*(H\Gamma)$ better for this example.

There are further unclear examples in \cite[Sect. 3.2.5]{zenner}.

\textbf{(9)} The representation theory of hypergraph $C^*$-algebras is not understood. See 
    \cite[Sect. 3.3]{zenner} for some very first steps. See also the theses \cite{zenner, trieb, schaeferthesis} for more work on hypergraph $C^*$-algebras.

\bibliographystyle{alpha}
\bibliography{references}

\newpage



\section{\textbf{Appendix A --- List of graph and hypergraph $C^*$-algebras}}
\label{Appendix_A}

We list some examples of graph and hypergraph $C^*$-algebras as overviews. We visualize them and use colored edges of different thickness to mark single edges if it simplifies the picture.

\begin{longtable}[h]{p{0.15\linewidth}|p{0.35\linewidth}|p{0.4\linewidth}}
\textbf{$C^*$-Algebra} & \textbf{Definition} & \textbf{Hypergraph} \\
\hline 
&&\\
$\C$ & 
$V=\{v\}$, $E=\emptyset$ & 
    \begin{minipage}{7cm}
    \centering
    \begin{tikzpicture}
            \node (v0) at (0,0) {};
            \fill (v0) circle[radius=2pt];
    \end{tikzpicture}
    \end{minipage}\\
&&\\
\hline
&&\\

  $M_n(\C)$

&   $V=\{v_1,\dots, v_n\}$ \newline
    $E=\{e_1,\dots, e_{n-1}\}$ \newline
    $s(e_j)=v_{j+1}$, $r(e)=v_j$ 
    
&    \begin{minipage}{7cm}
    \centering
    \begin{tikzpicture}[ > = stealth, auto, thick]
        \node (v0) at (0,0) {$v_1$};
        \node (v1) at (2,0) {$v_2$};
        \node (v2) at (4,0) {$\dots$};
        \node (v3) at (6,0) {$v_n$};
        \path[->] (v1) edge (v0); 
        \node (e0) at (1,0.25) {$e_1$};
        \path[->] (v2) edge (v1); 
        \node (e1) at (3,0.25) {$e_2$};
        \path[->] (v3) edge (v2); 
        \node (e2) at (5,0.25) {$e_{n-1}$};
        \node (e3) at (1,0.75) {};
    \end{tikzpicture}
    \end{minipage}\\


&   $V=\{v,w\}$\newline
    $E=\{e_1,\dots, e_{n-1}\}$ \newline
    $s(e_j)=v$, $r(e_j)=w$
 
&   \begin{minipage}{7cm}
    \centering 
    \begin{tikzpicture}[ > = stealth, auto, thick]
            \node (v0) at (0,0) {$v$};
            \node (v1) at (2,0) {$w$};
            \path[->, bend right=40] (v0) edge (v1); 
            \node (e0) at (1,0.7) {$e_1$};
            \path[->, bend left=40] (v0) edge (v1);
            \node (e1) at (1,-0.75) {$e_{n-1}$};
            \node (e2) at (1,0.05) {$\vdots$};
    \end{tikzpicture} 
    \end{minipage}\\


&    $V=\{v,w_1,\dots, w_{n-1}\}$ \newline
     $E=\{e_1,\dots, e_{n-1}\}$ \newline
     $s(e_j)=v$, $r(e_j)=w_j$ 

&   \begin{minipage}{7cm}
    \centering     
    \begin{tikzpicture}[ > = stealth, auto, thick]
            \node (v2) at (5,0) {$v$};
            \node (v3) at (7.2,1) {$w_1$};
            \node (v4) at (7,0) {$\;\;\vdots$};
            \node (v5) at (7.4,-1) {$w_{n-1}$};
            \path[->] (v2) edge (v3); 
            \node (e0) at (5.8,0.7) {$e_1$};
            \path[->, dotted] (v2) edge (v4);
            \path[->] (v2) edge (v5); 
            \node (e0) at (5.8,-0.8) {$e_{n-1}$};
    \end{tikzpicture}
    \end{minipage}
    \\
&&\\    
\hline


$\mathcal O_n$

Cuntz algebra, see also Prop. \ref{hypercuntz}

&    $V=\{v\}$ \newline
     $E=\{e_1,\dots, e_n\}$ \newline
     $s(e_j)=v, r(e_j)=v$ 

&   \begin{minipage}{7cm}
    \centering    
    \begin{tikzpicture}[ > = stealth, auto, thick]
            \clip(-0.5,-1.5) rectangle (3,1.5);
            \node (v0) at (0,0) {$v$};
            \path[->,in=-50,out=50,loop,scale=2] (v0) edge (v0); 
            \node (e0) at (1.25,0) {$e_1$};
            \node (e0) at (1.85,0) {$\dots$};
            \path[->,in=-50,out=50,loop,scale=9] (v0) edge (v0); 
            \node (e0) at (2.7,0) {$e_{n}$};
    \end{tikzpicture}  
    \end{minipage}\\

&   \newline 
    $V=\{v_1,\dots, v_n\}$ \newline
     $E=\{e_1,\dots, e_n\}$ \newline
     $s(e_j)=v_j, r(e_j)=\{v_1, \dots, v_n\}$ 

&  \begin{minipage}{7cm}
    \centering    
    \begin{tikzpicture}[ > = stealth, auto, thick]
            \clip(-1.5,-4) rectangle (5,3);
            \node (v0) at (0,0) {$v_1$};
            \node (v1) at (2,2) {$v_2$};
            \node (v2) at (4,0) {$\dots$};
            \node (v3) at (2,-2) {$v_n$};
            \path[->,out=210,in=150,loop,scale=1, blue, line width=1.7pt] (v0) edge (v0); 
            \node (e0) at (0.4,1.3) {$e_1$};
            \path[->,out=120,in=60,loop,scale=1, green, line width=1.3pt] (v1) edge (v1); 
            \node (e0) at (3.5,1.35) {$e_{2}$};
            \path[->,out=300,in=240,loop,scale=1, orange, line width=0.4pt] (v3) edge (v3); 
            \node (e0) at (0.4,-1.25) {$e_{n}$};
            \path[->, bend left=20, blue, line width=1.7pt] (v0) edge (v1);
            \path[->, bend left=20, green, line width=1.1pt] (v1) edge (v0);
            \path[->, bend left=20, blue, line width=1.7pt] (v0) edge (v2);
            \path[->, bend left=20, dotted] (v2) edge (v0);
            \path[->, bend left=20, blue, line width=1.7pt] (v0) edge (v3);
            \path[->, bend left=20, orange, line width=0.4pt] (v3) edge (v0);
            \path[->, bend left=20, green, line width=1.1pt] (v1) edge (v2);
            \path[->, bend left=20, dotted] (v2) edge (v1);
            \path[->, bend left=20, green, line width=1.1pt] (v1) edge (v3);
            \path[->, bend left=20, orange, line width=0.4pt] (v3) edge (v1);
            \path[->, bend left=20, dotted] (v2) edge (v3);
            \path[->, bend left=20, orange, line width=0.4pt] (v3) edge (v2);
    \end{tikzpicture}  
    \end{minipage}\\

\hline

$\mathcal{T}$

Toeplitz algebra, see also Prop. \ref{toeplitz}

&    $V=\{v,w\}$ \newline
     $E=\{e,f\}$ \newline
     $s(e)=w, r(e)=v$ \newline
     $s(f)=w, r(f)=w$

&   \begin{minipage}{7cm}
    \centering    
    \begin{tikzpicture}[ > = stealth, auto, thick]
        \node (v0) at (0,0) {$v$};
        \node (v1) at (2,0) {$w$};
        \path[->] (v1) edge (v0); 
        \node (e0) at (1,0.25) {$e$};
        \path[->,in=-50,out=50,loop,scale=3] (v1) edge (v1); 
        \node (e0) at (3.1,0) {$f$};
    \end{tikzpicture}  
    \end{minipage}\\

&    $V=\{v,w\}$ \newline
     $E=\{e,f\}$ \newline
     $s(e)=\{w\}, r(e)=\{v,w\}$

&   \begin{minipage}{7cm}
    \centering    
    \begin{tikzpicture}[ > = stealth, auto, thick]
        \node (v0) at (0,0) {$v$};
        \node (v1) at (2,0) {$w$};
        \path[->] (v1) edge (v0); 
        \node (e0) at (1,0.25) {$e$};
        \path[->,in=-50,out=50,loop,scale=3] (v1) edge (v1); 
        \node (e0) at (3.1,0) {$e$};
    \end{tikzpicture} 
    \end{minipage}\\


&    $V=\{v,w\}$ \newline
     $E=\{e,f\}$ \newline
     $s(e)=\{v,w\}, r(e)=\{w\}$

&   \begin{minipage}{7cm}
    \centering    
    \begin{tikzpicture}[ > = stealth, auto, thick]
        \node (v0) at (0,0) {$v$};
        \node (v1) at (2,0) {$w$};
        \path[->] (v0) edge (v1); 
        \node (e0) at (1,0.25) {$e$};
        \path[->,in=-50,out=50,loop,scale=3] (v1) edge (v1); 
        \node (e0) at (3.1,0) {$e$};
    \end{tikzpicture} 
    \end{minipage}\\

\hline
&&\\
$M_2(C(\T))$

&    $V=\{v,w\}$ \newline
     $E=\{e,f\}$ \newline
     $s(e)=v, r(e)=w$ \newline
     $s(f)=w, r(f)=v$

&   \begin{minipage}{7cm}
    \centering    
    \begin{tikzpicture}[ > = stealth, auto, thick]
        \node (v0) at (0,0) {$v$};
        \node (v1) at (2,0) {$w$};
        \path[->, bend left=30] (v0) edge (v1); 
        \node (e0) at (1,0.6) {$e$};
        \path[->, bend left=30] (v1) edge (v0);
        \node (e0) at (1,-0.6) {$f$};
    \end{tikzpicture} 
    \end{minipage}\\

\hline
&&\\
$C(S^1)*\C^n$

see also Prop. \ref{productgraph}

&    $V=\{v_1,\dots, v_n\}$ \newline
     $E=\{e\}$ \newline
     $s(e)=\{v_1,\dots,v_n\}$\newline
     $r(e)=\{v_1,\dots,v_n\}$ 

&   \begin{minipage}{7cm}
    \centering    
    \begin{tikzpicture}[ > = stealth, auto, thick]
            \node (v1) at (1.5,1.5) {$v_1$};
            \node (v2) at (3,1) {$v_2$};
            \node (v3) at (3,-1) {$\dots$};
            \node (v4) at (1.5,-1.5) {$v_{n-1}$};
            \node (v5) at (0,0) {$v_n$};
            \path[-] (v1) edge (v2); 
            \path[-] (v1) edge (v3); 
            \path[-] (v1) edge (v4); 
            \path[-] (v1) edge (v5); 
            \path[-] (v2) edge (v3); 
            \path[-] (v2) edge (v4); 
            \path[-] (v2) edge (v5); 
            \path[-] (v3) edge (v4); 
            \path[-] (v3) edge (v5); 
            \path[-] (v4) edge (v5); 
    \end{tikzpicture} 
    \end{minipage}\\
&&\\
\hline
&&\\
$\mathcal O_2*\C^n$

see also Prop. \ref{productgraph2}
&    $V=\{v_1,\dots, v_n\}$ \newline
     $E=\{e_1, \dots, e_m\}$ \newline
     $s(e_j)=\{v_1,\dots,v_n\}$\newline
     $r(e_j)=\{v_1,\dots,v_n\}$ 

&   \begin{minipage}{7cm}
    \centering    
    \begin{tikzpicture}[ > = stealth, auto, thick]
            \node (v1) at (1.5,1.5) {$v_1$};
            \node (v2) at (3,1) {$v_2$};
            \node (v3) at (3,-1) {$\dots$};
            \node (v4) at (1.5,-1.5) {$v_{n-1}$};
            \node (v5) at (0,0) {$v_n$};
            \path[-, bend left=10] (v1) edge (v2); 
            \path[-, bend right=10] (v1) edge (v2);
            \path[-, bend left=10] (v1) edge (v3); 
            \path[-, bend right=10] (v1) edge (v3); 
            \path[-, bend left=10] (v1) edge (v4); 
            \path[-, bend right=10] (v1) edge (v4); 
            \path[-, bend left=10] (v1) edge (v5); 
            \path[-, bend right=10] (v1) edge (v5); 
            \path[-, bend left=10] (v2) edge (v3); 
            \path[-, bend right=10] (v2) edge (v3); 
            \path[-, bend left=10] (v2) edge (v4); 
            \path[-, bend right=10] (v2) edge (v4); 
            \path[-, bend left=10] (v2) edge (v5); 
            \path[-, bend right=10] (v2) edge (v5); 
            \path[-, bend left=10] (v3) edge (v4); 
            \path[-, bend right=10] (v3) edge (v4); 
            \path[-, bend left=10] (v3) edge (v5); 
            \path[-, bend right=10] (v3) edge (v5); 
            \path[-, bend left=10] (v4) edge (v5); 
            \path[-, bend right=10] (v4) edge (v5); 
    \end{tikzpicture} 
    \end{minipage}\\

\end{longtable}

\newpage

\section{\textbf{Appendix B --- List of non-amenable hypergraphs}}
\label{Appendix_Non_Amenable_hypergraphs}

In the following we list a bunch of non-amenable hypergraphs. Since we have to  ensure that the remaining quotient is non-nuclear, $n$ must be chosen sufficiently large. The crucial non-nuclear part of the hypergraph is colored blue.

\begin{longtable}[h]{p{0.35\linewidth}p{0.4\linewidth}}
(B.1)\newline
           $V:=\{v_1,\dots v_n\}$\newline
            $E:=\{f\}$ \newline
           $s(f):=\{v_1,\dots, v_n\}$\newline 
           $r(f):=\{v_1,\dots, v_n\}$
 
&    \begin{minipage}{8cm}
    \centering
    \begin{tikzpicture}[ > = stealth, auto, thick, scale=0.8]
            \node (v1) at (1.5,1.5) {$v_1$};
            \node (v2) at (3,1) {$v_2$};
            \node (v3) at (3,-1) {$\dots$};
            \node (v4) at (1.5,-1.5) {$v_{n-1}$};
            \node (v5) at (0,0) {$v_n$};
            \path[-] (v1) edge (v2); 
            \path[-] (v1) edge (v3); 
            \path[-] (v1) edge (v4); 
            \path[-] (v1) edge (v5); 
            \path[-] (v2) edge (v3); 
            \path[-] (v2) edge (v4); 
            \path[-] (v2) edge (v5); 
            \path[-] (v3) edge (v4); 
            \path[-] (v3) edge (v5); 
            \path[-] (v4) edge (v5); 
        \end{tikzpicture}
    \end{minipage}\\

\hline

(B.2)\newline
            $V:=\{w, v_1,\dots v_n\}$\newline
            $E:=\{e, f\}$\newline
            $s(e):=\{w\}$\newline  
            $r(e):=\{v_n\}$\newline
            $s(f):=\{v_1,\dots, v_n\}$\newline
             $r(f):=\{v_1,\dots, v_n\}$
     
    & 
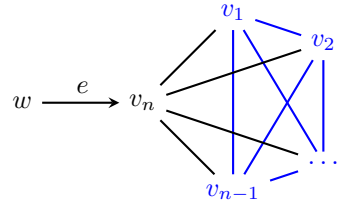
    \begin{minipage}{8cm}
    \centering
    \begin{tikzpicture}[ > = stealth, auto, thick, scale=0.8]
            \node (v0) at (0,0) {$w$};
            \node (v1) at (3.5,1.5) {$\textcolor{blue}{v_1}$};
            \node (v2) at (5,1) {$\textcolor{blue}{v_2}$};
            \node (v3) at (5,-1) {$\textcolor{blue}{\dots}$};
            \node (v4) at (3.5,-1.5) {$\textcolor{blue}{v_{n-1}}$};
            \node (v5) at (2,0) {$v_n$};
            \path[->] (v0) edge (v5); 
            \node (e0) at (1,0.25) {$e$};
            \path[-, blue] (v1) edge (v2); 
            \path[-, blue] (v1) edge (v3); 
            \path[-, blue] (v1) edge (v4); 
            \path[-] (v1) edge (v5); 
            \path[-, blue] (v2) edge (v3); 
            \path[-, blue] (v2) edge (v4); 
            \path[-] (v2) edge (v5); 
            \path[-, blue] (v3) edge (v4); 
            \path[-] (v3) edge (v5);  
            \path[-] (v4) edge (v5); 
        \end{tikzpicture}
         \end{minipage}\\

\hline
 
 (B.3)\newline 
            $V:=\{w_1, w_2, v_1,\dots v_n\}$\newline
            $E:=\{e, f\}$\newline
            $s(e):=\{w_1, w_2\}$\newline
             $r(e):=\{v_n\}$\newline
            $s(f):=\{v_1,\dots, v_n\}$\newline 
            $r(f):=\{v_1,\dots, v_n\}$
     
    & 
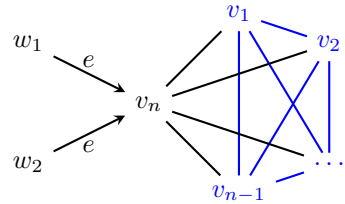
        \begin{minipage}{8cm}
    \centering
    \begin{tikzpicture}[ > = stealth, auto, thick, scale=0.8]
            \node (w1) at (0,1) {$w_1$};
            \node (w2) at (0,-1) {$w_2$};
            \node (v1) at (3.5,1.5) {$\textcolor{blue}{v_1}$};
            \node (v2) at (5,1) {$\textcolor{blue}{v_2}$};
            \node (v3) at (5,-1) {$\textcolor{blue}{\dots}$};
            \node (v4) at (3.5,-1.5) {$\textcolor{blue}{v_{n-1}}$};
            \node (v5) at (2,0) {$v_n$};
            \path[->] (w1) edge (v5);
            \node (e0) at (1,0.7) {$e$};
            \path[->] (w2) edge (v5);
            \node (e0) at (1,-0.7) {$e$};
            \path[-, blue] (v1) edge (v2); 
            \path[-, blue] (v1) edge (v3); 
            \path[-, blue] (v1) edge (v4); 
            \path[-] (v1) edge (v5); 
            \path[-, blue] (v2) edge (v3); 
            \path[-, blue] (v2) edge (v4); 
            \path[-] (v2) edge (v5); 
            \path[-, blue] (v3) edge (v4); 
            \path[-] (v3) edge (v5);  
            \path[-] (v4) edge (v5); 
        \end{tikzpicture}
               \end{minipage}\\

\hline

(B.4)\newline
            $V:=\{w, v_1,\dots v_n\}$\newline
            $E:=\{e, f\}$\newline
            $s(e):=\{w\}$\newline
             $r(e):=\{v_{n-1}, v_n\}$\newline
            $s(f):=\{v_1,\dots, v_n\}$\newline  
            $r(f):=\{v_1,\dots, v_n\}$
    
    &    
        \begin{minipage}{8cm}
    \centering
        \begin{tikzpicture}[ > = stealth, auto, thick, scale=0.8]
            \node (w1) at (1,-1.5) {$w$};
            \node (v1) at (3.5,1.5) {$\textcolor{blue}{v_1}$};
            \node (v2) at (5,1) {$\textcolor{blue}{v_2}$};
            \node (v3) at (5,-1) {$\textcolor{blue}{\dots}$};
            \node (v4) at (3.5,-1.5) {$v_{n-1}$};
            \node (v5) at (2,0) {$v_n$};
            \path[->] (w1) edge (v5);
            \node (e0) at (1.2,-0.6) {$e$};
            \path[->] (w1) edge (v4);
            \node (e0) at (2,-1.75) {$e$};
            \path[-, blue] (v1) edge (v2); 
            \path[-, blue] (v1) edge (v3); 
            \path[-] (v1) edge (v4); 
            \path[-] (v1) edge (v5); 
            \path[-, blue] (v2) edge (v3); 
            \path[-] (v2) edge (v4); 
            \path[-] (v2) edge (v5); 
            \path[-] (v3) edge (v4); 
            \path[-] (v3) edge (v5);  
            \path[-] (v4) edge (v5); 
        \end{tikzpicture}
                       \end{minipage}\\

\hline
 
 (B.5)\newline  
            $V:=\{w_1, w_2, v_1,\dots v_n\}$\newline
            $E:=\{e_1, e_2, f\}$\newline
            $s(e_1):=\{w, v_{n-1}\}$\newline
             $r(e_1):=\{v_{n-1}, v_n\}$\newline
            $s(e_2):=\{w, v_{n-1}\}$\newline 
            $r(e_2):=\{v_{n-1}, v_n\}$\newline
            $s(f):=\{v_1,\dots, v_n\}$\newline  
            $r(f):=\{v_1,\dots, v_n\}$
            
        &
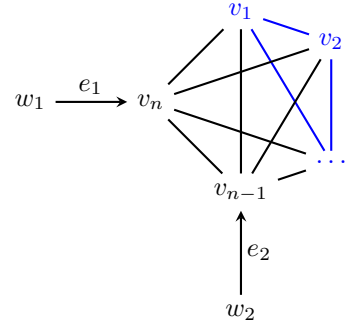
      \begin{minipage}{8cm}
    \centering
        \begin{tikzpicture}[ > = stealth, auto, thick, scale=0.8]
            \node (w1) at (0,0) {$w_1$};
            \node (w2) at (3.5,-3.5) {$w_2$};
            \node (v1) at (3.5,1.5) {$\textcolor{blue}{v_1}$};
            \node (v2) at (5,1) {$\textcolor{blue}{v_2}$};
            \node (v3) at (5,-1) {$\textcolor{blue}{\dots}$};
            \node (v4) at (3.5,-1.5) {$v_{n-1}$};
            \node (v5) at (2,0) {$v_n$};
            \path[->] (w1) edge (v5);
            \node (e0) at (1,0.25) {$e_1$};
            \path[->] (w2) edge (v4);
            \node (e0) at (3.8,-2.5) {$e_2$};
            \path[-, blue] (v1) edge (v2); 
            \path[-, blue] (v1) edge (v3); 
            \path[-] (v1) edge (v4); 
            \path[-] (v1) edge (v5); 
            \path[-, blue] (v2) edge (v3); 
            \path[-] (v2) edge (v4); 
            \path[-] (v2) edge (v5); 
            \path[-] (v3) edge (v4); 
            \path[-] (v3) edge (v5);  
            \path[-] (v4) edge (v5); 
        \end{tikzpicture}
                          \end{minipage}\\

\hline
  
  (B.6)\newline 
                $V:=\{v_1,\dots v_n\}$\newline
                $E:=\{e, f\} $\newline
                $s(e):=\{v_{n-1}\}$\newline  
                $r(e):=\{v_n\}$\newline
                $s(f):=\{v_1,\dots, v_n\}$\newline  
                $r(f):=\{v_1,\dots, v_n\}$
        &
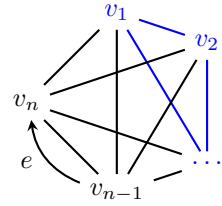
             \begin{minipage}{8cm}
    \centering
    \begin{tikzpicture}[ > = stealth, auto, thick, scale=0.8]
            \node (v1) at (3.5,1.5) {$\textcolor{blue}{v_1}$};
            \node (v2) at (5,1) {$\textcolor{blue}{v_2}$};
            \node (v3) at (5,-1) {$\textcolor{blue}{\dots}$};
            \node (v4) at (3.5,-1.5) {$v_{n-1}$};
            \node (v5) at (2,0) {$v_n$};
            \path[->, bend left=30] (v4) edge (v5);
            \node (e0) at (2,-1) {$e$};
            \path[-, blue] (v1) edge (v2); 
            \path[-, blue] (v1) edge (v3); 
            \path[-] (v1) edge (v4); 
            \path[-] (v1) edge (v5); 
            \path[-, blue] (v2) edge (v3); 
            \path[-] (v2) edge (v4); 
            \path[-] (v2) edge (v5); 
            \path[-] (v3) edge (v4); 
            \path[-] (v3) edge (v5);  
            \path[-] (v4) edge (v5); 
        \end{tikzpicture}
                              \end{minipage}\\

\hline

(B.7)\newline  
            $V:=\{w, v_1,\dots v_n\}$\newline
            $E:=\{e, f\}$\newline
            $s(e):=\{w, v_{n-1}\}$\newline
              $r(e):=\{v_n\}$\newline
            $s(f):=\{v_1,\dots, v_n\}$\newline 
            $r(f):=\{v_1,\dots, v_n\}$
        &
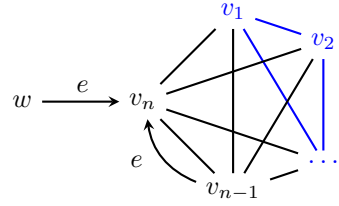
                     \begin{minipage}{8cm}
    \centering
        \begin{tikzpicture}[ > = stealth, auto, thick, scale=0.8]
            \node (w1) at (0,0) {$w$};
            \node (v1) at (3.5,1.5) {$\textcolor{blue}{v_1}$};
            \node (v2) at (5,1) {$\textcolor{blue}{v_2}$};
            \node (v3) at (5,-1) {$\textcolor{blue}{\dots}$};
            \node (v4) at (3.5,-1.5) {$v_{n-1}$};
            \node (v5) at (2,0) {$v_n$};
            \path[->] (w1) edge (v5);
            \node (e0) at (1,0.25) {$e$};
            \path[->, bend left=30] (v4) edge (v5); 
            \node (e0) at (1.9,-1) {$e$};
            \path[-, blue] (v1) edge (v2); 
            \path[-, blue] (v1) edge (v3); 
            \path[-] (v1) edge (v4); 
            \path[-] (v1) edge (v5); 
            \path[-, blue] (v2) edge (v3); 
            \path[-] (v2) edge (v4); 
            \path[-] (v2) edge (v5); 
            \path[-] (v3) edge (v4); 
            \path[-] (v3) edge (v5);  
            \path[-] (v4) edge (v5); 
        \end{tikzpicture}
                                     \end{minipage}\\

\hline
   
   (B.8)\newline   
            $V:=\{w, v_1,\dots v_n\}$\newline
            $E:=\{e, f\}$\newline
            $s(e):=\{w, v_{n-1}\}$\newline 
            $r(e):=\{v_{n-1}, v_n\}$\newline
            $s(f):=\{v_1,\dots, v_n\}$\newline  
            $r(f):=\{v_1,\dots, v_n\}$
            
        &
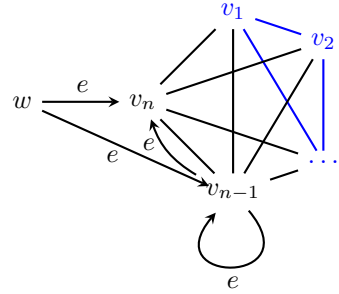
                     \begin{minipage}{8cm}
    \centering
        \begin{tikzpicture}[ > = stealth, auto, thick, scale=0.8]
            \node (w1) at (0,0) {$w$};
            \node (v1) at (3.5,1.5) {$\textcolor{blue}{v_1}$};
            \node (v2) at (5,1) {$\textcolor{blue}{v_2}$};
            \node (v3) at (5,-1) {$\textcolor{blue}{\dots}$};
            \node (v4) at (3.5,-1.5) {$v_{n-1}$};
            \node (v5) at (2,0) {$v_n$};
            \path[->] (w1) edge (v5);
            \node (e0) at (1,0.25) {$e$};
            \path[->] (w1) edge (3.1,-1.4);
            \node (e0) at (1.5,-0.9) {$e$};
            \path[->, bend left=20] (v4) edge (v5); 
            \node (e0) at (2.1,-0.7) {$e$};
            \path[->,in=230,out=310,loop,scale=1] (v4) edge (v4);
            \node (e0) at (3.5,-3) {$e$};
            \path[-, blue] (v1) edge (v2); 
            \path[-, blue] (v1) edge (v3); 
            \path[-] (v1) edge (v4); 
            \path[-] (v1) edge (v5); 
            \path[-, blue] (v2) edge (v3); 
            \path[-] (v2) edge (v4); 
            \path[-] (v2) edge (v5); 
            \path[-] (v3) edge (v4); 
            \path[-] (v3) edge (v5);  
            \path[-] (v4) edge (v5); 
        \end{tikzpicture}
                                          \end{minipage}\\

\hline
(B.9)\newline  
            $V:=\{w, v_1,\dots v_n\}$\newline
            $E:=\{e, f\}$\newline
            $s(e):=\{w\}$\newline 
            $r(e):=\{w\}$\newline
            $s(f):=\{w, v_1,\dots, v_n\}$\newline  
            $r(f):=\{v_1,\dots, v_n\}$
        &
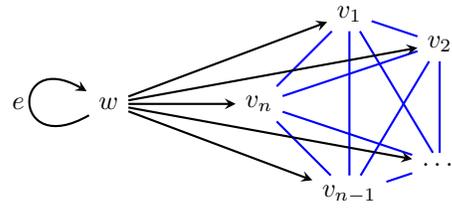
                   \begin{minipage}{8cm}
    \centering
        \begin{tikzpicture}[ > = stealth, auto, thick, scale=0.8]
            \node (v0) at (1.5,0) {$w$};
            \path[->,in=150,out=210,loop,scale=3] (v0) edge (v0); 
            \node (e0) at (0,0) {$e$};
            
            \node (v1) at (5.5,1.5) {$v_1$};
            \node (v2) at (7,1) {$v_2$};
            \node (v3) at (7,-1) {$\dots$};
            \node (v4) at (5.5,-1.5) {$v_{n-1}$};
            \node (v5) at (4,0) {$v_n$};
            \path[-, blue] (v1) edge (v2); 
            \path[-, blue] (v1) edge (v3); 
            \path[-, blue] (v1) edge (v4); 
            \path[-, blue] (v1) edge (v5); 
            \path[-, blue] (v2) edge (v3); 
            \path[-, blue] (v2) edge (v4); 
            \path[-, blue] (v2) edge (v5); 
            \path[-, blue] (v3) edge (v4); 
            \path[-, blue] (v3) edge (v5); 
            \path[-, blue] (v4) edge (v5); 
            
            \path[->] (v0) edge (v1); 
            \path[->] (v0) edge (v2); 
            \path[->] (v0) edge (v3); 
            \path[->] (v0) edge (v4); 
            \path[->] (v0) edge (v5); 
        \end{tikzpicture} 
                                                  \end{minipage}\\
      
\hline
 
 (B.10)\newline
            $V:=\{w, v_1,\dots v_n\}$\newline
            $E:=\{e, f\}$\newline
            $s(e):=\{w, v_n\}$\newline 
            $r(e_1):=\{w\}$\newline
            $s(f):=\{w, v_1,\dots, v_n\}$\newline  
            $r(f):=\{v_1,\dots, v_n\}$
        &
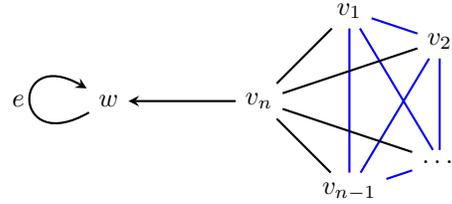
                 \begin{minipage}{8cm}
    \centering
        \begin{tikzpicture}[ > = stealth, auto, thick, scale=0.8]
            \node (v0) at (1.5,0) {$w$};
            \path[->,in=150,out=210,loop,scale=3] (v0) edge (v0); 
            \node (e0) at (0,0) {$e$};
            
            \node (v1) at (5.5,1.5) {$v_1$};
            \node (v2) at (7,1) {$v_2$};
            \node (v3) at (7,-1) {$\dots$};
            \node (v4) at (5.5,-1.5) {$v_{n-1}$};
            \node (v5) at (4,0) {$v_n$};
            \path[-, blue] (v1) edge (v2); 
            \path[-, blue] (v1) edge (v3); 
            \path[-, blue] (v1) edge (v4); 
            \path[-] (v1) edge (v5); 
            \path[-, blue] (v2) edge (v3); 
            \path[-, blue] (v2) edge (v4); 
            \path[-] (v2) edge (v5); 
            \path[-, blue] (v3) edge (v4); 
            \path[-] (v3) edge (v5);  
            \path[-] (v4) edge (v5); 
            
            \path[->] (v5) edge (v0); 

        \end{tikzpicture}        
                                                \end{minipage}\\
          
  \hline      
 
 (B.11)\newline
            $V:=\{w, v_1,\dots v_n\}$\newline
            $E:=\{e_1, e_2, e_3, f\}$\newline
            $s(e_1):=\{w\}$\newline
            $r(e_1):=\{w\}$\newline
            $s(e_2):=\{w\}$\newline
             $r(e_2):=\{v_n\}$\newline
            $s(e_3):=\{v_n\}$\newline 
            $r(e_3):=\{w\}$\newline
            $s(f):=\{w, v_1,\dots, v_n\}$\newline  
            $r(f):=\{v_1,\dots, v_n\}$
        &
               \begin{minipage}{8cm}
    \centering
        \begin{tikzpicture}[ > = stealth, auto, thick, scale=0.8]
            \node (v0) at (1.5,0) {$w$};
            \path[->,in=150,out=210,loop,scale=3] (v0) edge (v0); 
            \node (e0) at (-0.2,0) {$e_1$};
            
            \node (v1) at (5.5,1.5) {$v_1$};
            \node (v2) at (7,1) {$v_2$};
            \node (v3) at (7,-1) {$\dots$};
            \node (v4) at (5.5,-1.5) {$v_{n-1}$};
            \node (v5) at (4,0) {$v_n$};
            \path[-, blue] (v1) edge (v2); 
            \path[-, blue] (v1) edge (v3); 
            \path[-, blue] (v1) edge (v4); 
            \path[-] (v1) edge (v5); 
            \path[-, blue] (v2) edge (v3); 
            \path[-, blue] (v2) edge (v4); 
            \path[-] (v2) edge (v5); 
            \path[-, blue] (v3) edge (v4); 
            \path[-] (v3) edge (v5);  
            \path[-] (v4) edge (v5); 
            
            \path[->, bend left=20] (v5) edge (v0); 
            \node (e0) at (2.7,0.5) {$e_2$};
            \path[->, bend left=20] (v0) edge (v5); 
            \node (e0) at (2.7,-0.5) {$e_3$};
        \end{tikzpicture}     
                                                \end{minipage}
\end{longtable}

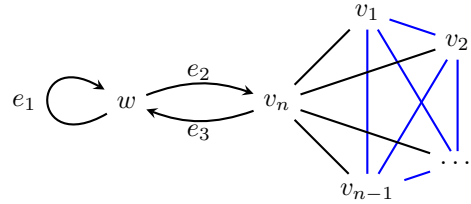
\end{document}